\documentclass[reqno]{amsart}

\usepackage{amsmath}
\usepackage{amsfonts}
\usepackage{amssymb,enumerate}
\usepackage{amsthm}
\usepackage[all]{xy}
\usepackage{rotating}
\usepackage{hyperref,color}
\theoremstyle{plain}
\newtheorem{lem}{Lemma}[section]
\newtheorem{cor}[lem]{Corollary}
\newtheorem{prop}[lem]{Proposition}
\newtheorem{thm}[lem]{Theorem}

\theoremstyle{definition}
\newtheorem{defn}[lem]{Definition}
\newtheorem{ex}[lem]{Example}

\newtheorem{disc}[lem]{Remark}

\newtheorem{notn}[lem]{Notation}
\newtheorem{fact}[lem]{Fact}

\newcommand{\cat}[1]{\mathcal{#1}}

\newcommand{\catd}{\cat{D}}

\newcommand{\pd}{\operatorname{pd}}

\newcommand{\id}{\operatorname{id}}	
\newcommand{\fd}{\operatorname{fd}}

\newcommand{\mspec}{\operatorname{m-Spec}}

\newcommand{\HH}{\operatorname{H}}
\newcommand{\Hom}{\operatorname{Hom}}	

\newcommand{\spec}{\operatorname{Spec}}
\newcommand{\s}{\mathfrak{S}}

\newcommand{\im}{\operatorname{Im}}
\newcommand{\shift}{\mathsf{\Sigma}}

\newcommand{\ideal}[1]{\mathfrak{#1}}
\newcommand{\m}{\ideal{m}}

\newcommand{\p}{\ideal{p}}

\newcommand{\fa}{\ideal{a}}
\newcommand{\fb}{\ideal{b}}

\newcommand{\comp}[1]{\widehat{#1}}

\newcommand{\ass}{\operatorname{Ass}}

\newcommand{\supp}{\operatorname{supp}}

\newcommand{\VE}{\operatorname{V}}

\newcommand{\bbz}{\mathbb{Z}}

\newcommand{\xra}{\xrightarrow}

\newcommand{\vf}{\varphi}

\renewcommand{\geq}{\geqslant}
\renewcommand{\leq}{\leqslant}

\newcommand{\Ext}[4][R]{\operatorname{Ext}_{#1}^{#2}(#3,#4)}	
\newcommand{\Rhom}[3][R]{\mathbf{R}\!\operatorname{Hom}_{#1}(#2,#3)}	
\newcommand{\Lotimes}[3][R]{#2\otimes^{\mathbf{L}}_{#1}#3}
\newcommand{\Otimes}[3][R]{#2\otimes_{#1}#3}
\renewcommand{\Hom}[3][R]{\operatorname{Hom}_{#1}(#2,#3)}

\newcommand{\LL}[2]{\mathbf{L}\Lambda^{\ideal{#1}}(#2)}
\newcommand{\LLL}[2]{\mathbf{L}\widehat\Lambda^{\ideal{#1}}(#2)}
\newcommand{\RG}[2]{\mathbf{R}\Gamma_{\ideal{#1}}(#2)}
\newcommand{\RRG}[2]{\mathbf{R}\widehat\Gamma_{\ideal{#1}}(#2)}

\newcommand{\Comp}[2]{\widehat{#1}^{\ideal{#2}}}

\newcommand{\catdfb}{\catd_{\text{b}}^{\text{f}}}
\newcommand{\catdb}{\catd_{\text{b}}}
\newcommand{\catdf}{\catd^{\text{f}}}
\newcommand{\RGa}[2]{\mathbf{R}\Gamma_{\mathfrak{#1}\Comp R{#1}}(#2)}
\newcommand{\LLa}[2]{\mathbf{L}\Lambda^{\mathfrak{#1}\Comp R{#1}}(#2)}
\newcommand{\LLno}[1]{\mathbf{L}\Lambda^{\ideal{#1}}}
\newcommand{\LLLno}[1]{\mathbf{L}\widehat\Lambda^{\ideal{#1}}}
\newcommand{\RGno}[1]{\mathbf{R}\Gamma_{\ideal{#1}}}
\newcommand{\RRGno}[1]{\mathbf{R}\widehat\Gamma_{\ideal{#1}}}
\newcommand{\RGano}[1]{\mathbf{R}\Gamma_{\mathfrak{#1}\Comp R{#1}}}
\newcommand{\LLano}[1]{\mathbf{L}\Lambda^{\ideal{#1}\Comp R{#1}}}

\newcommand{\catdator}{\catd(R)_{\text{$\fa$-tor}}}

\newcommand{\catdaator}{\catd(\Comp Ra)_{\text{$\fa\Comp Ra$-tor}}}

\newcommand{\compsa}{\comp S^{\fa S}}

\numberwithin{equation}{lem}

\begin{document}

\bibliographystyle{amsplain}

\author{Sean Sather-Wagstaff}

\address{Department of Mathematical Sciences,
Clemson University,
O-110 Martin Hall, Box 340975, Clemson, S.C. 29634
USA}

\email{ssather@clemson.edu}

\urladdr{https://ssather.people.clemson.edu/}

\thanks{
Sean Sather-Wagstaff was supported in part by a grant from the NSA}

\author{Richard Wicklein}

\address{Richard Wicklein, Mathematics and Physics Department, MacMurray College, 447 East College Ave., Jacksonville, IL 62650, USA}

\email{richard.wicklein@mac.edu}

\title{Adic semidualizing complexes}




\begin{abstract}
We introduce and study a class of objects that encompasses Christensen and Foxby's semidualizing modules and complexes and Kubik's quasi-dualizing modules: the class of $\mathfrak{a}$-adic semidualizing modules and complexes. 
We give examples and equivalent characterizations of these objects, including a characterization in terms of the more familiar 
semidualizing property.
As an application, we give a  proof of the existence of dualizing complexes over complete local rings that does not use the 
Cohen Structure Theorem.
\end{abstract}

\maketitle

\tableofcontents

\section{Introduction} \label{sec130805a}
Throughout this paper let $R$ be a commutative noetherian ring, let $\fa \subsetneq R$ be a proper ideal of $R$, and let $\Comp{R}{a}$ be the $\fa$-adic completion of $R$.
Let $K$ denote the Koszul complex over $R$ on a finite generating sequence for $\fa$.

\

This work is part 5 in a series of papers on derived local cohomology and derived local homology.
It builds on our previous papers~\cite{sather:afbha,sather:afcc,sather:elclh,sather:scc}, and it is applied 
in the paper~\cite{sather:afc}.

Duality is a powerful tool in many areas of mathematics. 
For instance, over a complete  Cohen-Macaulay local ring, Grothendieck's local duality~\cite{hartshorne:lc} uses Matlis duality to relate local cohomology modules
to Ext-modules (i.e., derived dual-modules) with respect to the ring's canonical module. 
When the ring is not Cohen-Macaulay, Grothendieck~\cite{hartshorne:rad} uses a dualizing complex to obtain similar\footnote{or, depending on your perspective, the same} conclusions.
This allows one to study local cohomology by studying Ext, and vice versa, which is incredibly useful.

Because of this and many other applications, dualizing complexes have become important in commutative algebra and algebraic geometry. 
The standard proof of the existence of a dualizing complex for a complete local ring $R$ uses the powerful Cohen Structure Theorem~\cite{cohen:sitclr}:
one surjects onto $R$ with a complete regular local ring $Q$, takes an injective resolution $I$ of $Q$ over itself, and shows that the complex
$\Hom[Q]RI$ is dualizing for $R$. 

One consequence of our work in the current paper is the following alternate construction of a dualizing complex
which avoids the Cohen Structure Theorem; see Theorem~\ref{cor150525axx}\eqref{cor150525axx1} below.

\begin{thm}\label{thm151203a}
Assume that $(R,\m,k)$ is local with $E_R(k)$ the injective hull of $k$ over $R$.
Let $F$ be a flat resolution of $E$ over $R$. 
Then the $\m$-adic completion $\Comp Fm$ is a dualizing complex over $\Comp Rm$.
\end{thm}

With the power of dualizing complexes in mind, much work has been devoted to the identification of good objects for use in dualities. 
For instance, Foxby~\cite{foxby:gmarm} introduced the ``PG modules of rank 1'' now known as \emph{semidualizing modules};
these are the finitely generated (i.e., noetherian) $R$-modules $C$ such that the natural homothety map $\chi\colon R\to\Hom CC$ given by
$\chi(r)(c):=rc$ is an isomorphism and $\Ext iCC=0$ for all $i\geq 1$.
Canonical modules over Cohen-Macaulay local rings are special cases of these. 
Christensen~\cite{christensen:scatac} extended this to the ``semidualizing complexes'', 
a notion that is flexible enough to encompass both the semidualizing modules as well as the dualizing complexes.
This theory is very useful, capturing not only the dualizing complexes, but also Avramov and Foxby's~\cite{avramov:rhafgd} ``relative dualizing complexes'', 
but it misses other important situations, e.g., Matlis duality.

The work of Kubik~\cite{kubik:qdm} begins to fill this gap by introducing the ``quasi-dualizing modules'' over a local ring $(R,\m,k)$: an artinian $R$-module $T$ 
is $\m$-torsion, so it is a module over  $\Comp Rm$, and $T$
is \emph{quasi-dualizing} if the natural homothety map $\Comp Rm\to\Hom TT$ is an isomorphism, and $\Ext iTT=0$ for all $i\geq 1$.
This includes the injective hull $E_R(k)$, i.e., the base for Matlis duality, as a special case.
However, this does not include the semidualizing objects as special cases, unless the ring is artinian and local, though it does come tantalizingly close,
with the same Ext-vanishing condition, a similar endomorphism algebra isomorphism, and a related finiteness conditions.

The primary goal of this paper is to fill this gap completely by introducing a single notion that recovers all the aforementioned examples as special cases:
that of an ``$\fa$-adic semidualizing complex''.
The general definition is necessarily somewhat technical, building on our papers~\cite{sather:afbha,sather:afcc,sather:elclh, sather:scc}
as well as the established literature on semidualizing complexes; see Definition~\ref{def120925e}. For modules, though, the definition is more straightforward:
an $\fa$-torsion $R$-module $M$ has the structure of a module over $\Comp Ra$, and $M$
is \emph{$\fa$-adically semidualizing} if 
$\Ext{i}{R/\mathfrak{a}}{M}$ is finitely generated for all $i$,  the natural homothety map $\Comp Ra\to\Hom MM$ is an isomorphism, and $\Ext iMM=0$ for all $i\geq 1$.
In particular, the special case $\fa=0$ recovers the semidualizing modules, and the maximal ideal $\fa=\m$ of a local ring yields the quasi-dualizing modules;
see Propositions~\ref{prop130528b} and~\ref{prop120925b}. 

Section~\ref{sec151008a} of this paper is devoted to the foundational properties of $\fa$-adic semidualizing complexes, with the help of
some preparatory lemmas from Section~\ref{sec130818b}. 
The main result of Section~\ref{sec151008a} is Theorem~\ref{lem151005a}, a characterization of the adic semidualizing property.
It shows first that any isomorphism $\Comp Ra\cong\Hom MM$ implies that the homothety map $\Comp Ra\to\Hom MM$ is an isomorphism,
which is somewhat surprising since $\Comp Ra$ and $M$ are not assumed to be finitely generated, a crucial feature of the definition.
Second, it characterizes this property in terms of a semidualizing condition over $\Comp Ra$. 
We state a partial version for modules here for perspective. 

\begin{thm}\label{thm151203b}
Let $M$ be an $R$-module with flat resolution $F$. Then the following conditions are equivalent.
\begin{enumerate}[\rm(i)]
\item
$M$ is $\fa$-adically semidualizing over $R$.
\item
$M$ is $\fa$-torsion, the module
$\Ext{i}{R/\mathfrak{a}}{M}$ is finitely generated for all $i$,  and one has $\Comp Ra\cong\Hom MM$  and $\Ext iMM=0$ for all $i\geq 1$.
\item
$M$ is $\fa$-torsion, and the completed complex $\Comp Fa$ is semidualizing over $\Comp Ra$.
\end{enumerate}
\end{thm}

The bulk of this paper is Section~\ref{sec150527a}, which is devoted to describing the connections between various flavors of semidualizing objects,
though one can already see hints of this in Theorem~\ref{thm151203b}. As a sample, the next result contains parts of Theorems~\ref{thm151012a} 
and~\ref{thm140204a}; see Remark~\ref{disc151015a} for a diagrammatic representation of this and more. 

\begin{thm}\label{thm151203c}
The following sets are in natural bijection:
\begin{enumerate}[\rm(a)]
\item the set of shift-isomorphism classes of $\fa$-adic semidualizing $R$-complexes,
\item the set of shift-isomorphism classes of semidualizing $\Comp Ra$-complexes, and
\item the set of shift-isomorphism classes of $\fa\Comp Ra$-adic semidualizing $\Comp Ra$-complexes.
\end{enumerate}
\end{thm}

Another result worth mentioning from this section is Theorem~\ref{cor151122a}, which states that the adic semidualizing property is local.

The concluding Section~\ref{sec151204a} contains Theorem~\ref{thm151203a} and other results about dualizing complexes.
It also includes characterizations of the adic semidualizing complexes in the case where they should all be trivial in some sense: when $R$ is Gorenstein.

While we have phrased much of this introduction in terms of modules, the bulk of this paper deals with the more general situation of chain complexes.
Specifically, we work primarily in the derived category.
Section~\ref{sec140109b} below contains some background material on this topic.

This work is largely inspired by the papers mentioned above, especially those of Christensen and Foxby~\cite{christensen:scatac,foxby:gmarm}.
We explore other $\fa$-adic aspects of these works in our subsequent paper~\cite{sather:afc}.

\section{Background}\label{sec140109b} 

\subsection*{Derived Categories}
We work in the derived category $\catd(R)$ with objects  the $R$-complexes
indexed homologically
$X=\cdots\to X_i\to X_{i-1}\to\cdots$;
see~\cite{hartshorne:rad,verdier:cd,verdier:1}.
Isomorphisms in $\catd(R)$ are identified by the symbol $\simeq$.
The $n$th shift (or suspension) of $X$ is denoted $\shift^nX$.
We also consider the next full triangulated subcategories:
\begin{align*}
\catd_+(R)
&\text{: objects are homologically bounded below $R$-complexes}\\
\catd_-(R)
&\text{: objects are homologically bounded above $R$-complexes}\\
\catdb(R)
&\text{: objects are homologically bounded  $R$-complexes}\\
\catdf(R)
&\text{: objects are homologically degree-wise finite $R$-complexes}
\end{align*}
Intersections of these categories are designated with two ornaments, e.g., $\catdfb(R)=\catdb(R)\bigcap\catdf(R)$.

\subsection*{Resolutions}
An $R$-complex $F$ is \emph{semi-flat}\footnote{In the literature, semi-flat complexes are sometimes called ``K-flat'' or ``DG-flat''.} 
if the functor $\Otimes -F$ respects injective quasiisomorphisms, that is, if each module $F_i$ is flat over $R$
and the functor $\Otimes -F$ respects quasiisomorphisms.
A \emph{semi-flat resolution} of an $R$-complex $X$ is a quasiisomorphism $F\xra\simeq X$ such that $F$ is semi-flat;
for $X\in\catdb(R)$, the \emph{flat dimension} $\fd_R(X)$ is the length of the shortest bounded semi-flat resolution of $X$,
if one exists:
$$\fd_R(X):=\inf\{\sup\{i\in\bbz\mid F_i\neq 0\}\mid \text{$F\xra\simeq X$ is a semi-flat resolution}\}.$$
The \emph{injective} and \emph{projective} versions of these notions are defined similarly.

For the following items, consult~\cite[Section 1]{avramov:hdouc} or~\cite[Chapters 3 and~5]{avramov:dgha}.
Bounded below complexes of flat $R$-modules are semi-flat, and 
bounded above complexes of injective $R$-modules are semi-injective. 
Every $R$-complex admits a semi-flat resolution (hence, a semi-projective one) and a semi-injective resolution.

\subsection*{Derived Functors}
The right derived functor of Hom is $\Rhom --$, which is computed via a semi-projective resolution in the first slot
or a semi-injective resolution in the second slot. 
The left derived functor of tensor product is $\Lotimes --$, which is computed via a semi-flat resolution in either slot.

Local cohomology and local homology, described next, play a major role in this work.
These notions go back to Grothendieck~\cite{hartshorne:lc}, and Matlis~\cite{matlis:kcd,matlis:hps}, respectively;
see also~\cite{lipman:lhcs,lipman:llcd}.
Let $\Lambda^{\fa}(-)$ denote the $\fa$-adic completion functor, and
$\Gamma_{\fa}(-)$ is the $\fa$-torsion functor, i.e.,
for an $R$-module $M$ we have
$$\Lambda^{\fa}(M)=\Comp Ma
\qquad
\qquad
\qquad
\Gamma_{\fa}(M)=\{ x \in M \mid \fa^{n}x=0 \text{ for } n \gg 0\}.$$ 
A module $M$ is \textit{$\fa$-torsion} if $\Gamma_{\fa}(M)=M$.

The associated left and right derived functors (i.e., \emph{derived local homology and cohomology} functors)
are  $\LL a-$ and $\RG a-$.
Specifically, given an $R$-complex $X\in\catd(R)$ with a semi-flat resolution $F\xra\simeq X$ and a 
semi-injective resolution $X\xra\simeq I$,  we have $\LL aX\simeq\Lambda^{\fa}(F)$ and $\RG aX\simeq\Gamma_{\fa}(I)$.
Note that these definitions yield natural transformations $\RGno a\to\id\to \LLno a$, induced by the natural morphisms
$\Gamma_{\fa}(I)\to I$ and $F\to \Lambda^{\fa}(F)$.
Let $\catdator$ denote the full subcategory of $\catd(R)$ of all complexes $X$ such that the morphism
$\RG aX\to X$ is an isomorphism.

The definitions of $\RG aX$ and $\LL aX$ yield complexes over the completion $\Comp Ra$, and we denote by
$\LLLno a$ and $\RRGno a$ the associated functors $\catd(R)\to\catd(\Comp Ra)$.

\begin{fact}\label{fact130619b}
If $Q\colon \catd(\Comp Ra)\to\catd(R)$ is the forgetful functor, then it follows readily that
$Q\circ\LLLno a\simeq\LLno a$ and $Q\circ\RRGno a\simeq\RGno a$.
If $X\in\catdf_+(R)$, then there is a natural isomorphism
$\LL aX\simeq \Lotimes{\Comp Ra}{X}$ by~\cite[Proposition 2.7]{frankild:volh}.
Moreover, the proof of this result shows that there is a natural isomorphism
$\LLL aX\simeq \Lotimes{\Comp Ra}{X}$ in $\catd(\Comp Ra)$.\footnote{This 
is based on the fact that, for a finitely generated free $R$-module $L$, induction on the rank of $L$
shows that the natural isomorphism
$\Otimes{\Comp Ra}{L}\cong\Comp La$ is $\Comp Ra$-linear.} 
From~\cite[Theorem~(0.3) and Corollary~(3.2.5.i)]{lipman:lhcs}, there are natural isomorphisms
\begin{align*}
\RG a-\simeq\Lotimes{\RG aR}{-}&&
\LL a-\simeq\Rhom{\RG aR}{-}.
\end{align*}
More generally, by~\cite[Theorems~3.2 and~3.6]{shaul:hccac} we have
\begin{align*}
\RRG a-\simeq\Lotimes{\RRG aR}{-}&&
\LLL a-\simeq\Rhom{\RRG aR}{-}.
\end{align*}
\end{fact}

Here is an important feature of these constructions, sometimes called MGM equivalence (after Matlis, Greenlees, and May).

\begin{fact}\label{fact150626a}
From~\cite[Corollary to Theorem (0.3)$^*$]{lipman:lhcs}
and~\cite[Theorem~1.2]{yekutieli:hct} the following natural morphisms 
are isomorphisms:
\begin{align*}
\RGno a\circ\id\xra\simeq\RGno a\circ\LLno a
&&\LLno a\circ\RGno a\xra\simeq\LLno a\circ\id \\
\RGno a\circ\RGno a\xra\simeq\id\circ\RGno a
&&\id\circ\LLno a\xra\simeq\LLno a\circ\LLno a.
\end{align*}
\end{fact}

The following notion of support is due to Foxby~\cite{foxby:bcfm}.

\begin{defn}\label{defn130503a}
Let $X\in\catd(R)$.
The \emph{small support} of $X$ is
\begin{align*}
\operatorname{supp}_R(X)
&=\{\mathfrak{p} \in \operatorname{Spec}(R)\mid \Lotimes{\kappa(\p)}X\not\simeq 0 \} 
\end{align*}
where $\kappa(\p):=R_\p/\p R_\p$.
\end{defn}

\begin{fact}\label{cor130528a}
Let $X\in\catd(R)$. Then 
we know that $\supp_R(X)\subseteq\VE(\fa)$ if and only if $X\in\catdator$ if and only if each homology module $\HH_i(X)$ is $\fa$-torsion,
by~\cite[Proposition~5.4]{sather:scc}
and~\cite[Corollary~4.32]{yekutieli:hct}.
\end{fact}

The next fact and definition take their cues from work of 
Hartshorne~\cite{hartshorne:adc},
Kawasaki~\cite{kawasaki:ccma,kawasaki:ccc}, and
Melkersson~\cite{melkersson:mci}.

\begin{fact}[\protect{\cite[Theorem~1.3]{sather:scc}}]
\label{thm130612a}
For $X\in\catd_{\text b}(R)$, the next conditions are equivalent.
\begin{enumerate}[\rm(i)]
\item\label{cor130612a1}
One has $\Lotimes{K^R(\underline{y})}{X}\in\catdfb(R)$  for some (equivalently for every) generating sequence $\underline{y}$ of $\fa$.
\item\label{cor130612a2}
One has  $\Lotimes{X}{R/\mathfrak{a}}\in\catd^{\text{f}}(R)$.
\item\label{cor130612a3}
One has  $\Rhom{R/\mathfrak{a}}{X}\in\catd^{\text{f}}(R)$.
\item\label{cor130612a4}
One has $\LLL aX\in\catdfb(\Comp Ra)$.
\end{enumerate}
\end{fact}

\begin{defn}\label{def120925d}
An $R$-complex $X\in\catdb(R)$ is \emph{$\mathfrak{a}$-adically finite} if it satisfies the equivalent conditions of Fact~\ref{thm130612a} and $\operatorname{supp}_R(X) \subseteq \operatorname{V}(\mathfrak{a})$.
\end{defn}

\begin{disc}\label{disc151204a}
Because of Fact~\ref{cor130528a},
an $R$-module $M$ is $\fa$-adically finite if and only if it is $\fa$-torsion and has $\Ext i{R/\fa}M$ finitely generated for all $i$.
\end{disc}

\begin{ex}\label{ex160206a}
Let $X\in\catdb(R)$ be given.
\begin{enumerate}[(a)]
\item \label{ex160206a1}
If $X\in\catdfb(R)$, then we have $\supp_R(X)=\VE(\fb)$ for some ideal $\fb$, and it follows that $X$ is $\fa$-adically finite
whenever $\fa\subseteq\fb$. (The case $\fa=0$ is from~\cite[Proposition~7.8(a)]{sather:scc}, and the general case follows readily.)
\item \label{ex160206a2}
$K$ and $\RG aR$ are $\fa$-adically finite, by~\cite[Fact~3.4 and Theorem~7.10]{sather:scc}.
\item \label{ex160206a3}
The homology modules of $X$ are artinian if and only if there is an ideal $\fa$ of finite colength (i.e., such that $R/\fa$ is artinian)
such that $X$ is $\fa$-adically finite, by~\cite[Proposition~5.11]{sather:afcc}.
\end{enumerate}
\end{ex}

We continue with a few  semidualizing definitions.

\begin{defn}\label{defn151204a}
An $R$-complex $C\in\catdfb(R)$ is \emph{semidualizing} if the natural homothety morphism $\chi^R_C\colon R\to\Rhom CC$ is an isomorphism in $\catd(R)$.
The set of shift-isomorphism classes of semidualizing $R$-complexes is denoted $\s(R)$.
A \emph{tilting} $R$-complex\footnote{These are called ``invertible''  in~\cite{avramov:rrc1}.} is a semidualizing $R$-complex of finite projective dimension,
and a \emph{dualizing} $R$-complex is a semidualizing $R$-complex of finite injective dimension.
\end{defn}

We end this section with a few useful notes about completions.

\begin{lem}\label{lem151205a}
Let $\psi\colon R\to\Comp Ra$ be the natural map.
\begin{enumerate}[\rm(a)]
\item\label{lem151205a1}
There is a bijection $\mspec(R)\bigcap\VE(\fa)\to\mspec(\Comp Ra)$ given by $\m\mapsto\m\Comp Ra$.
The inverse of this bijection is given by contraction along $\psi$. 
\item\label{lem151205a2}
There is a bijection $\VE(\fa)\to\VE(\fa\Comp Ra)$ given by $\p\mapsto\p\Comp Ra$.
The inverse of this bijection is given by contraction along $\psi$. 
\item\label{lem151205a4}
If $R$ is locally Gorenstein, then so is $\Comp Ra$. 
\end{enumerate}
\end{lem}

\begin{proof}
\eqref{lem151205a1}--\eqref{lem151205a2}
The induced map $R/\fa\to \Comp Ra/\fa\Comp Ra$ is an isomorphism.
Since $\fa\Comp Ra$ is contained in the Jacobson radical of $\Comp Ra$, the result now follows readily.

\eqref{lem151205a4}
Assume that $R$ is locally Gorenstein, and let $\ideal M\in\mspec(\Comp Ra)\subseteq\VE(\fa\Comp Ra)$ be given.
By part~\eqref{lem151205a1}, the contraction $\m$ of $\ideal M$ in $R$ is maximal and satisfies $\m\Comp Ra=\ideal M$. 
It follows that the closed fibre of the induced flat, local ring homomorphism $R_\m\to\Comp Ra_{\ideal M}$ is a field. 
Thus, the assumption that $R_\m$ is Gorenstein implies that $\Comp Ra_{\ideal M}$ is Gorenstein as well. 
\end{proof}

\section{Homothety Morphisms}\label{sec130818b}

This section is devoted to some technical lemmas that we use to show that $\fa$-adic semidualizing complexes are well-defined.

\begin{lem}\label{lem151008a}
Let $M\in\catd_-(R)$
with $\supp_R(M)\subseteq V(\fa)$. 
Then $M$ has a bounded above semi-injective resolution $M\xra\simeq J$ over $R$ consisting of 
$\Comp{R}{a}$-module homomorphisms of injective $\Comp Ra$-modules.
\end{lem}

\begin{proof}
From~\cite[Proposition~3.8 and Corollary~3.9]{sather:scc} we know that $M$ has a bounded above semi-injective resolution $M\xra\simeq J$ over $R$ consisting of 
$\fa$-torsion injective $R$-modules with $\ass_R(J_i)\subseteq\supp_R(M)\subseteq\VE(\fa)$ for each $i$.
The torsion condition implies that each differential $\partial^J_i$ is $\Comp Ra$-linear and  the natural map
$J\to\Otimes{\Comp Ra}{J}$ is an isomorphism; see~\cite[Fact~2.1 and Lemma~2.2]{kubik:hamm2}. 
The associated prime condition implies that each prime ideal $\p\in\ass_R(J_i)$  satisfies
$\Comp Ra/\p\Comp Ra\cong R/\p$, so we have $\Otimes{\Comp Ra}{\kappa(\p)}\cong\kappa(\p)$. 
We conclude from~\cite[Theorem~1]{foxby:imufbc} that each $J_i\cong\Otimes{\Comp Ra}{J_i}$ is injective over $\Comp Ra$,
as desired.
\end{proof}

\begin{disc}\label{disc151204b}
One might be tempted to prove the preceding result for arbitrary (that is, unbounded) complexes as follows. 
Let $X\xra\simeq J$ be a semi-injective resolution.
The condition $\supp_R(X)\subseteq\VE(\fa)$ says that the natural morphism $\RG aX\to X$ is an isomorphism in $\catd(R)$,
so we have $X\simeq\Gamma_{\fa}(J)$. The complex $\Gamma_{\fa}(J)$ consists of $\fa$-torsion injective $R$-modules,
so the isomorphism gives a resolution of the desired form, as in the proof of the preceding result. 
The problem with this line of reasoning is that $\Gamma_\fa(J)$ can fail to be semi-injective; see~\cite[Example~3.1]{sather:elclh}.
\end{disc}

\begin{defn}\label{rmk140129a}
Let $M\in\catd_-(R)$
with $\supp_R(M)\subseteq V(\fa)$. 
Let $M\xra\simeq J$ be a semi-injective resolution as in Lemma~\ref{lem151008a}.
This yields a well-defined chain map $\chi^{\Comp Ra}_J\colon\Comp Ra\to\Hom JJ$ given by $\chi^{\Comp Ra}_J(r)(j)=rj$.
This in turn gives rise to a well-defined ``homothety morphism'' $\chi^{\Comp Ra}_M\colon\Comp Ra\to\Rhom MM$ in $\catd(R)$.
\end{defn}

The rest of this section is devoted to a lemma for use in the proof of Theorem~\ref{thm151203b} from the introduction.
A subtlety of the result is worth noting here: 
in part~\eqref{lem151004a1} we only have an isomorphism over $R$; however, we are able to translate it to information about $\Comp Ra$-isomorphisms.

\begin{lem}\label{lem151004a}
Let $M\in\catd_-(R)$ be such that $\supp_R(M)\subseteq\VE(\fa)$.
\begin{enumerate}[\rm(a)]
\item\label{lem151004a1}
One has
$\Rhom MM\simeq\Rhom[\Comp Ra]{\LLL aM}{\LLL aM}$
in $\catd(R)$. 
\item\label{lem151004a2}
There is an isomorphism
$\Comp Ra\simeq \Rhom{M}{M}$  in $\catd(R)$
if and only if there is an isomorphism
$\Comp Ra\simeq \Rhom[\Comp Ra]{\LLL aM}{\LLL aM}$ in $\catd(\Comp Ra)$.
\item\label{lem151004a3}
The  morphism
$\chi^{\Comp Ra}_{M}\colon \Comp Ra\to \Rhom{M}{M}$ is an isomorphism in $\catd(R)$
if and only if the  morphism
$\chi^{\Comp Ra}_{\LLL aM}\colon \Comp Ra\to \Rhom[\Comp Ra]{\LLL aM}{\LLL aM}$ is an isomorphism in $\catd(\Comp Ra)$.
\end{enumerate}
\end{lem}

\begin{proof}
By Lemma~\ref{lem151008a}, the $R$-complex $M$ has a bounded above semi-injective resolution $M\xra\simeq I$ consisting of 
injective $\fa$-torsion $\Comp Ra$-modules. This explains the first three steps in the next display. 
\begin{align*}
\Rhom MM
&\simeq\Hom JJ \\
&=\Hom{\Gamma_\fa(J)}{\Gamma_\fa(J)}\\
&=\Hom[\Comp Ra]{\Gamma_\fa(J)}{\Gamma_\fa(J)}\\
&\simeq\Rhom[\Comp Ra]{\RRG aM}{\RRG aM}
\end{align*}
The fourth step follows from the fact that $J$ is a bounded above semi-injective resolution of $M$,
because this implies that $\Gamma_\fa(J)$ is a bounded above complex of injective $\Comp Ra$-modules,
so it is a semi-injective resolution of $\RRG aM$ over $\Comp Ra$.

\eqref{lem151004a1}
In~\cite[Theorem~4.7]{sather:elclh} we show that there is an isomorphism
\begin{equation}\label{eq151008a}
\Rhom[\Comp Ra]{\RRG aM}{\RRG aM}\simeq\Rhom[\Comp Ra]{\LLL aM}{\LLL aM}
\end{equation}
in $\catd(\Comp Ra)$. With the isomorphisms described above, this explains the  isomorphism in part~\eqref{lem151004a1} of the theorem.

\eqref{lem151004a2}
The isomorphisms from the first paragraph of this proof provide the first step in the next sequence in $\catd(\Comp Ra)$.
\begin{align*}
\Rhom{\Comp Ra}{\Rhom MM}
&\simeq\Rhom{\Comp Ra}{\Rhom[\Comp Ra]{\RRG aM}{\RRG aM}}\\
&\simeq\Rhom[\Comp Ra]{\Lotimes{\Comp Ra}{Q(\RRG aM)}}{\RRG aM}\\
&\simeq\Rhom[\Comp Ra]{\Lotimes{\Comp Ra}{\RG aM}}{\RRG aM}\\
&\simeq\Rhom[\Comp Ra]{\RRG aM}{\RRG aM}\\
&\simeq\Rhom[\Comp Ra]{\Comp Ra}{\Rhom[\Comp Ra]{\RRG aM}{\RRG aM}}
\end{align*}
The second step here is Hom-tensor adjointness, where $Q$ is the forgetful functor $\catd(\Comp Ra)\to\catd(R)$.
The third and fourth steps are from Fact~\ref{fact130619b}, 
and the last one is from Hom-cancellation.
From this sequence, we have $\Comp Ra\simeq\Rhom MM$ in $\catd(R)$
if and only if $\Comp Ra\simeq\Rhom[\Comp Ra]{\RRG aM}{\RRG aM}$ in $\catd(\Comp Ra)$,
i.e.,
if and only if $\Comp Ra\simeq\Rhom[\Comp Ra]{\LLL aM}{\LLL aM}$ in $\catd(\Comp Ra)$, by~\eqref{eq151008a}.

\eqref{lem151004a3}
The isomorphisms from the first paragraph of this proof also yield the next commutative diagram in $\catd(R)$.
$$\xymatrix{
\Comp Ra\ar[r]^-{\chi^{\Comp Ra}_{\RRG aM}}\ar[d]_{\chi^{\Comp Ra}_{M}}
&\Rhom[\Comp Ra]{\RRG aM}{\RRG aM} \ar[ld]^{\simeq}\\
\Rhom MM
}$$
In particular, the morphism $\chi^{\Comp Ra}_{M}$ is an isomorphism in $\catd(R)$
if and only if $\chi^{\Comp Ra}_{\RRG aM}$ is so, that is, if and only if $\chi^{\Comp Ra}_{\RRG aM}$ is an isomorphism in $\catd(\Comp Ra)$,
since $\chi^{\Comp Ra}_{\RRG aM}$ is a morphism in $\catd(\Comp Ra)$.
Thus, to explain the desired bi-implication, it suffices to show that the  homothety morphisms
$\chi^{\Comp Ra}_{\RRG aM}\colon \Comp Ra\to \Rhom[\Comp Ra]{\RRG aM}{\RRG aM}$
and 
$\chi^{\Comp Ra}_{\LLL aM}\colon \Comp Ra\to \Rhom[\Comp Ra]{\LLL aM}{\LLL aM}$
in $\catd(\Comp Ra)$ are isomorphisms simultaneously.
To this end, first consider the natural $\Comp Ra$-isomorphisms
$$\RRG aM\simeq\RRG a{\LL aM}\simeq\RGa a{\LLL aM}$$ 
from~\cite[Lemmas~4.4--4.5]{sather:elclh}.
It follows that $\chi^{\Comp Ra}_{\RRG aM}$ is an isomorphism in $\catd(\Comp Ra)$ if and only if
$\chi^{\Comp Ra}_{\RGa a{\LLL aM}}$ is so. 
Similarly, the isomorphism 
$$\LLL aM\simeq\LLa a{\LLL aM}$$
from~\cite[Theorem~4.3]{sather:elclh} shows that $\chi^{\Comp Ra}_{\LLL aM}$ is an isomorphism in $\catd(\Comp Ra)$ if and only if
$\chi^{\Comp Ra}_{\LLa a{\LLL aM}}$ is so. 
Thus, it remains to show that $\chi^{\Comp Ra}_{\RGa a{\LLL aM}}$ is an isomorphism in $\catd(\Comp Ra)$ if and only if
$\chi^{\Comp Ra}_{\LLa a{\LLL aM}}$ is so. 
The fact that these are isomorphisms simultaneously follows from the next commutative diagram in $\catd(\Comp Ra)$, wherein we set $N:=\LLL aM$:
$$\xymatrix{
\Comp Ra\ar[r]^-{\chi^{\Comp Ra}_{\RGa a{N}}}\ar[d]_{\chi^{\Comp Ra}_{\LLa a{N}}}
&\Rhom[\Comp Ra]{\RGa a{N}}{\RGa a{N}}\ar[d]_\simeq^{(\varepsilon_{\fa\Comp Ra}^N)_*}
\\
\Rhom[\Comp Ra]{\LLa a{N}}{\LLa a{N}}\ar[d]^\simeq_{(\vartheta^{\fa\Comp Ra}_N)^*}
&\Rhom[\Comp Ra]{\RGa a{N}}{N}\ar[d]_\simeq^{(\vartheta^{\fa\Comp Ra}_N)_*}
\\
\Rhom[\Comp Ra]{N}{\LLa a{N}}\ar[r]_-\simeq^-{(\varepsilon_{\fa\Comp Ra}^N)^*}
&\Rhom[\Comp Ra]{\RGa a{N}}{\LLa a{N}}.
}$$
The isomorphisms in this diagram are from~\cite[Theorem (0.3)$^*$]{lipman:lhcs}.\footnote{See 
also~\cite[Thoerem~6.12]{yekutieli:hct}. In addition, 
we have~\cite[Remark~6.14]{yekutieli:hct} for a discussion of some aspects of this result, and~\cite{yekutieli:ehct} for a correction.}
\end{proof}

\section{Adic Semidualizing Complexes}\label{sec151008a}

This section consists of examples and fundamental properties of $\fa$-adic semidualizing complexes, including the proof of Theorem~\ref{thm151203b} from the introduction. 

\begin{defn}\label{def120925e}
An $R$-complex $M$ is \emph{$\mathfrak{a}$-adically semidualizing} if $M$ is $\mathfrak{a}$-adically finite 
(see Definition~\ref{def120925d}) and the homothety morphism 
$\chi_{M}^{\widehat{R}^{\mathfrak{a}}}\colon \widehat{R}^{\mathfrak{a}} \rightarrow \Rhom{M}{M}$ 
from Definition~\ref{rmk140129a} is an isomorphism in $\catd(R)$.
The set of shift-isomorphism classes in $\catd(R)$ of $\fa$-adically semidualizing complexes is denoted
$\s^{\fa}(R)$.
\end{defn}

\begin{disc}\label{defn130503c}
An $R$-module $M$ is $\fa$-adically semidualizing as an $R$-complex
if it is $\fa$-adically finite\footnote{See Remark~\ref{disc151204a}.}, the natural homothety map $\Comp{R}{a} \to \Hom{M}{M}$, 
defined as in Defintion~\ref{rmk140129a},
is an isomorphism, and $\Ext{i}{M}{M}=0$ for all $i \geq 1$.
\end{disc}

\begin{disc}\label{rmk140204a}
If $M$ is an $\fa$-adically semidualizing $R$-complex then  $\supp_R(M)=\VE(\fa)$, by~\cite[Proposition~7.17]{sather:scc}.
\end{disc}

The next two propositions show that Definition~\ref{def120925e} yields the semidualizing complexes and quasi-dualizing modules as special cases. 

\begin{prop}\label{prop130528b}
An $R$-complex $M$ is semidualizing if and only if it is $0$-adically semidualizing,
that is, we have $\s(R)=\s^0(R)$. 
\end{prop}

\begin{proof}
The $R$-complex $M$ is 0-adically finite if and only if it is in $\catdfb(R)$; see, e.g., \cite[Proposition~7.8(a)]{sather:scc}.
Because of the isomorphism $\comp R^0\cong R$, we see that the homothety morphisms 
$\chi^{\comp R^0}_M$ and $\chi^R_M$ are isomorphisms simultaneously.
Thus, the result follows by definition.
\end{proof}

\begin{prop}\label{prop120925b}
Assume that $(R,\m)$ is local. 
\begin{enumerate}[\rm(a)]
\item\label{prop120925b1}
An $R$-complex $M\in\catdb(R)$ is $\m$-adically semidualizing if and only if each homology module $\HH_i(M)$ is artinian 
and the homothety morphism $\chi_{M}^{\widehat{R}^{\mathfrak{m}}}\colon \widehat{R}^{\mathfrak{m}} \rightarrow \Rhom{M}{M}$ is an isomorphism in $\catd(R)$.
\item\label{prop120925b2}
An $R$-module $T$ is quasi-dualizing if and only if it is $\m$-adically semidualizing.
\item\label{prop120925b3}
The injective hull $E:=E_R(R/\m)$ is $\m$-adically semidualizing.
\end{enumerate}
\end{prop}

\begin{proof}
Since each $\HH_i(M)$ is artinian if and only if $M\in\catdb(R)$
is $\m$-adically finite by \cite[Proposition 7.8(b)]{sather:scc}, the result follows by definition and the
standard isomorphism $\Comp Rm\xra\cong\Hom EE$.
\end{proof}

In light of Fact~\ref{cor130528a} and Remark~\ref{disc151204a}, the next result contains Theorem~\ref{thm151203b} from the introduction.

\begin{thm}\label{lem151005a}
Let $M\in\catdb(R)$. The following conditions are equivalent:
\begin{enumerate}[\rm(i)]
\item \label{lem151005a1}
$M$ is $\fa$-adically semidualizing over $R$;
\item \label{lem151005a2}
$M$ is $\fa$-adically finite, and we have $\Comp{R}{a} \simeq \Rhom{M}{M}$ in $\catd(R)$; and
\item \label{lem151005a3}
$\supp_R(M)\subseteq\VE(\fa)$ and $\LLL aM$ is semidualizing over $\Comp Ra$.
\end{enumerate}
\end{thm}

\begin{proof}
The implication \eqref{lem151005a1}$\implies$\eqref{lem151005a2} is by definition.

\eqref{lem151005a2}$\implies$\eqref{lem151005a3}
Assume that $M$ is $\fa$-adically finite, and $\Comp{R}{a} \simeq \Rhom{M}{M}$ in $\catd(R)$.
By definition, this implies that $\supp_R(M)\subseteq\VE(\fa)$ and $\LLL aM\in\catdfb(\Comp Ra)$.
Also, we have $\Comp Ra\simeq\Rhom[\Comp Ra]{\LLL aM}{\LLL aM}$ in $\catd(\Comp Ra)$ by 
Lemma~\ref{lem151004a}\eqref{lem151004a2},
so~\cite[Proposition~3.1]{avramov:rrc1} implies that $\LLL aM$ is semidualizing over $\Comp Ra$.

\eqref{lem151005a3}$\implies$\eqref{lem151005a1}
This is verified like the previous implication, using Lemma~\ref{lem151004a}\eqref{lem151004a3}.
\end{proof}

The next result and its corollary show how to build examples of adic semidualizing complexes. 

\begin{thm}\label{thm140109a}
If $M$ is $\fa$-adically semidualizing complex over $R$ and $\fb$ is an ideal of $R$, then 
$\RG{\fb}{M}\simeq\RG{a+b}{M}$ is $\fa+\fb$-adically semidualizing.
In particular, if $\fa\subseteq\fb$, then
$\RG{\fb}{M}$ is $\fb$-adically semidualizing.
\end{thm}

\begin{proof}
Our assumptions imply that $\supp_R(M)\subseteq\VE(\fa)$, so the natural morphism $\RG aM\to M$ is an isomorphism
in $\catd(R)$ by Fact~\ref{cor130528a}. It follows that we have the following isomorphisms in $\catd(R)$:
$$\RG bM\simeq\RG{b}{\RG aM}\simeq\RG{a+b}M.$$
Thus, we replace $\fb$ with $\fa+\fb$ to assume that $\fa\subseteq\fb$.

By~\cite[Theorem~7.10]{sather:scc}, the fact that $M$ is $\fa$-adically semidualizing implies that $\RG{\fb}{M}$ is $\fb$-adically finite.
Thus,  it suffices by Theorem~\ref{lem151005a} to show that $\Comp{R}{b}\simeq\Rhom{\RG{b}{M}}{\RG{b}{M}}$ in $\catd(R)$. 

Since $\Comp Ra$ is flat over $R$,  the first isomorphism in the next sequence is by definition:
$$
\textbf{L}\Lambda^{\fb}(\Comp{R}{a})
\simeq\Lambda^{\fb}(\Comp{R}{a})
\simeq\Comp{R}{b}.
$$
The second isomorphism follows from the containment $\fa\subseteq\fb$ since $\Lambda^{\fb}(-)=\Comp{(-)}b$.
This explains the first isomorphism in $\catd(R)$ in the next sequence.
\begin{align*}
\Comp Rb
&\simeq\textbf{L}\Lambda^{\fb}(\Comp{R}{a}) \\
&\simeq  \textbf{L}\Lambda^{\fb}(\Rhom MM)\\
&\simeq\Rhom{\RG bR}{\Rhom MM} \\
&\simeq\Rhom{\Lotimes{M}{\RG bR}}{M} \\
&\simeq\Rhom{\RG bM}{M} \\
&\simeq\Rhom{\RG bM}{\RG bM}
\end{align*}
The second isomorphism follows from  the isomorphism $\widehat{R}^{\mathfrak{a}} \simeq \Rhom MM$.
The third and fifth isomorphisms are by Fact~\ref{fact130619b}, and the fourth one is Hom-tensor adjointness.
The last isomorphism is a consequence of~\cite[Theorem (0.3)$^*$]{lipman:lhcs}.
\end{proof}

\begin{cor}\label{thm130330b}
If $C$ is a semidualizing  $R$-complex, then the complex $\RG{\fa}{C}$ is $\fa$-adically semidualizing.
Hence, the complex $\RG{\fa}{R}$ is $\fa$-adically semidualizing.
\end{cor}

\begin{proof}
Since  ``semidualizing'' is equivalent to ``$0$-adically semidualizing'', by Proposition~\ref{prop130528b}, the first conclusion follows from Theorem~\ref{thm140109a}.
The second conclusion is the special case $C=R$.
\end{proof}

\begin{disc}\label{disc151016b}
Alternately, one can obtain Corollary~\ref{thm130330b}
from  MGM equivalence~\ref{fact150626a}, as follows.
By~\cite[Theorem~7.10]{sather:scc}, we know that $\RG aC$ is $\fa$-adically finite, 
so it suffices by Theorem~\ref{lem151005a} to show that $\Comp{R}{a}\simeq\Rhom{\RG{a}{C}}{\RG{a}{C}}$ in $\catd(R)$. 
MGM equivalence provides the first isomorphism
in the following sequence:
$$\LL a{\RG aC}\simeq\LL aC\simeq\Lotimes{\Comp Ra}C.$$
The second isomorphism is from Fact~\ref{fact130619b}.
This explains the second isomorphism in the next sequence:
\begin{align*}
\Rhom{\RG aC}{\RG aC}
&\simeq\Rhom{C}{\LL a{\RG aC}}\\
&\simeq\Rhom{C}{\Lotimes{\Comp Ra}C}\\
&\simeq\Lotimes{\Comp Ra}{\Rhom{C}{C}}\\
&\simeq\Lotimes{\Comp Ra}{R}\\
&\simeq \Comp Ra.
\end{align*}
The first isomorphism follows from Fact~\ref{fact130619b} with Hom-tensor adjointness.
The third isomorphism is tensor-evaluation~\cite[Lemma 4.4(F)]{avramov:hdouc},
the fourth one is from the assumption $\Rhom CC\simeq R$,
and the fifth one is tensor-cancellation.
From this perspective, the fact that $\RG aR$ is $\fa$-adically semidualizing is even easier:
\begin{align*}
\Rhom{\RG aR}{\RG aR}
&\simeq\Rhom{R}{\LL a{\RG aR}}\\
&\simeq\LL a{\RG aR}\\
&\simeq\LL aR\\
&\simeq\Comp Ra.\qedhere
\end{align*}
\end{disc}

\section{Transfer of the Adic Semidualizing Property}\label{sec150527a}

This  section focuses on some transfer properties for adic semidualizing complexes, including Theorems~\ref{thm151203a} and~\ref{thm151203c} from the introduction.
In particular, it furthers the theme from Theorem~\ref{lem151005a}, which describes some of the interplay
between the semidualizing $\Comp Ra$-complexes and the $\fa$-adically semidualizing $R$-complexes.

\begin{notn}\label{notn151011a}
In this section, let $\vf\colon R\to S$ be a  homomorphism of commutative noetherian rings with $\fa S\neq S$,
and consider the forgetful functor $Q\colon\catd(S)\to\catd(R)$.
\end{notn}

\subsection*{Restriction of Scalars}

Note that each result of this subsection and the next one holds for the natural flat homomorphism $R\to\Comp Ra$,
moreover, for the map $R\to\Comp Rb$ for any ideal $\fb\subseteq\fa$.

\begin{prop}
\label{prop151011a}
Assume that $\vf$ is flat  and that the induced map $\Comp Ra\to\comp S^{\fa S}$ is an isomorphism.
Then an $S$-complex $Y\in\catd(S)$ is $\fa S$-adically semidualizing over $S$
if and only if $Q(Y)$ is $\fa$-adically semidualizing over $R$.
\end{prop}

\begin{proof}
From~\cite[Lemma~5.3]{sather:afcc}, we know that $\supp_R(Q(Y))\subseteq\VE(\fa)$ if and only if
$\supp_S(Y)\subseteq\VE(\fa S)$.
Thus, we assume for the rest of this proof that $\supp_S(Y)\subseteq\VE(\fa S)$.
It follows that $Y$ is a $\fa S$-adically semidualizing over $S$ if and only if 
$\mathbf{L}\widehat\Lambda^{\mathfrak{a}S}(Y)$ is semidualizing over $\comp S^{\fa S}\cong\Comp Ra$, by Theorem~\ref{lem151005a};
and $Q(Y)$ is $\fa$-adically semidualizing over $R$ if and only if $\LLL a{Q(Y)}$ is semidualizing over $\Comp Ra$.
Thus, it suffices to show that we have $\mathbf{L}\widehat\Lambda^{\mathfrak{a}S}(Y)\simeq\LLL a{Q(Y)}$
in $\catd(\comp S^{\fa S})$.\footnote{Technically, we should  use the forgetful functor
$\catd(\comp S^{\fa S})\to\catd(\Comp Ra)$ here. However, since our isomorphism assumption implies that this is an equivalence,
we avoid the extra notation.}

Let $F\xra\simeq Y$ be a semi-flat resolution over $S$. 
Since $S$ is flat over $R$, this also yields a semi-flat resolution $Q(F)\xra\simeq Q(Y)$. 
Completing an $S$-complex with respect to $\fa$ is the same as completing it with respect to $\fa S$, 
so we have 
$$\LLL a{Q(Y)}\simeq\Lambda^{\fa}(Q(F))\cong\Lambda^{\fa S}(F)\simeq\mathbf{L}\widehat\Lambda^{\mathfrak{a}S}(Y)$$
in $\catd(\comp S^{\fa S})$, as desired.
\end{proof}

\begin{disc}\label{disc151011a}
It is worth noting that, even when the map $\vf$ is flat and local, the hypothesis $\Comp Ra\cong\comp S^{\fa S}$
is necessary for each implication in the previous result.
Indeed, using the natural maps $\vf$ where
$S=R[\![X]\!]$ or $R[\![X]\!]/(X^2)$, one implication fails for $Y=S$, and the other implication fails with $Y=R$.
\end{disc}

\subsection*{Extension of Scalars}

Our next Theorem is akin to~\cite[Theorem~5.6]{christensen:scatac} and~\cite[Theorem~4.5]{frankild:rrhffd}, 
which describe finite flat dimension base change for semidualizing complexes. 
First, we prove two lemmas.
Recall that the homomorphism $\vf$ is \emph{locally of finite flat dimension} if, for every prime $\ideal P\in\spec(S)$
the induced map $\vf_{\ideal P}\colon R_{\p}\to S_{\ideal P}$ has finite flat dimension where $\p$ is the contraction of $\ideal P$ in $R$.
The main point for introducing this notion is the following fact: if $\vf$ is of finite flat dimension,
(more generally, if it is locally of finite flat dimension--see the next lemma)
then the induced map $\Comp Ra\to\compsa$ is locally of finite flat dimension,
though it is not clear that it has finite flat dimension; see~\cite[Theorem~6.11(c)]{avramov:cmporh}.
\newcommand{\fP}{\ideal P}

\begin{lem}\label{lem151012a}
\begin{enumerate}[\rm(a)]
\item\label{lem151012a1}
If $\fd_R(S)<\infty$, then $\vf$ is locally of finite flat dimension.
\item\label{lem151012a2}
If for every maximal ideal $\ideal M\in\mspec(S)$
the induced map $\vf_{\ideal M}$ has finite flat dimension, then $\vf$ is locally of finite flat dimension.
\end{enumerate}
\end{lem}

\begin{proof}
\eqref{lem151012a1}
Let $\ideal P\in\spec(S)$, and let $\p\in\spec(R)$ denote the contraction of $\ideal P$ in $R$.
We have $\fd_{R_\p}(S_\p)\leq\fd_R(S)<\infty$.
Since the induced local map $\vf_{\ideal P}\colon R_{\p}\to S_{\ideal P}$ is the composition of the natural maps $R_\p\to S_\p\to S_{\ideal P}$,
each of which finite flat dimension,
the map $\vf_{\ideal P}$ has finite flat dimension as well.

\eqref{lem151012a2}
Assume that for every maximal ideal $\ideal M\in\mspec(S)$
the induced map $\vf_{\ideal M}$ has finite flat dimension.
Given a prime $\ideal P\in\spec(S)$, let $\ideal M\in\mspec(S)$ be such that $\ideal P\subseteq\ideal M$.
The induced map $\vf_{\ideal M}$ has finite flat dimension by assumption, so it is locally of finite flat dimension by part~\eqref{lem151012a1}.
Under the prime correspondence for localization, the map $\vf_{\ideal P}$ corresponds to the map
$(\vf_{\ideal M})_{\ideal P_{\ideal M}}$, so it has finite flat dimension, as desired.
\end{proof}

For perspective and use in the next results, note that~\cite[Proposition~5.6(a)]{sather:afcc} shows that $\supp_R(S)=\im(\vf^*)$ where
$\vf^*\colon\spec(S)\to\spec(R)$ is the induced map.

\begin{lem}\label{lem151012b}
Assume that $\vf$ is locally of finite flat dimension, and let $C\in\catdf(R)$.
\begin{enumerate}[\rm(a)]
\item\label{lem151012b1}
If $C$ is a semidualizing $R$-complex such that $\Lotimes SC\in\catdb(S)$, 
then $\Lotimes SC$ is semidualizing over~$S$.
\item\label{lem151012b2}
The converse of part~\eqref{lem151012b1} holds when $\supp_R(S)\supseteq\mspec(R)$.
\end{enumerate}
\end{lem}

\begin{proof}
\eqref{lem151012b1}
Assume that $C$ is a semidualizing $R$-complex such that $\Lotimes SC\in\catdb(S)$.
Since this implies that $C\in\catdfb(R)$, it follows that we have $\Lotimes SC\in\catdfb(S)$.
The semidualizing property is local for such complexes~\cite[Lemma~2.3]{frankild:rrhffd}, 
so it suffices to show that $(\Lotimes SC)_{\ideal P}$ is semidualizing
over $S_{\ideal P}$ for each
prime $\ideal P\in\spec(S)$. In other words, 
we need to show that the following $S_{\ideal P}$-complex is semidualizing. 
$$\Lotimes[S]{S_{\fP}}{(\Lotimes SC)}\simeq\Lotimes[R_{\p}]{S_{\fP}}{(\Lotimes{R_{\p}}C)}$$
This is so by~\cite[Theorem~4.5]{frankild:rrhffd}, because the maps $R\to R_{\p}\to S_{\fP}$ each have finite flat dimension,
where $\p$ is the contraction of $\ideal P$ in $R$.

\eqref{lem151012b2}
Assume that
$\Lotimes SC$ is semidualizing over~$S$, and that we have $\supp_R(S)\supseteq\mspec(R)$.
In particular, this implies that $\Lotimes SC\in\catdb(R)$.

Claim: $C\in\catdb(R)$.\footnote{Note that if we had $\fd_R(S)<\infty$, this would follow from~\cite[Theorem~I]{frankild:rrhffd}.}
For this, we use a bit of bookkeeping notation from Foxby~\cite{foxby:ibcahtm}:
given an $R$-complex $Z$, set
\begin{align*}
\sup(Z)&:=\sup\{ i\in\bbz\mid\HH_i(Z)\neq 0\}\\
\inf(Z)&:=\inf\{ i\in\bbz\mid\HH_i(Z)\neq 0\}.
\end{align*}
with the conventions $\sup\emptyset=-\infty$ and $\inf\emptyset=\infty$.
Thus, to prove the claim, we need to show that $-\infty<\inf(C)$ and $\sup(C)<\infty$.
Let $\m\in\mspec(R)\subseteq\supp_R(S)=\vf^*(\spec(S))$. It follows that there is a maximal ideal
$\ideal M\in\mspec(S)$ such that $\m=\vf^{-1}(\ideal M)$. 
The induced map $\vf_{\ideal M}$ has finite flat dimension, so~\cite[Theorem~I(c)]{frankild:rrhffd}
explains the first step in the next sequence.
\begin{align*}
\inf(C_\m)
&=\inf(\Lotimes[R_\m]{S_{\ideal M}}{C_\m})
=\inf((\Lotimes SC)_{\ideal M})
\geq\inf(\Lotimes SC)
\end{align*} 
The other steps are straightforward. 
It follows that we have
$$\inf(C)=\inf\{\inf(C_\m)\mid\m\in\mspec(R)\}\geq\inf(\Lotimes SC)>-\infty.$$
For $\sup(C)$, we argue similarly:
\begin{gather*}
\sup(C_\m)\leq\sup(\Lotimes[R_\m]{S_{\ideal M}}{C_\m})
=\sup((\Lotimes SC)_{\ideal M})
\leq\sup(\Lotimes SC)
\\
\sup(C)=\sup\{\sup(C_\m)\mid\m\in\mspec(R)\}\leq\sup(\Lotimes SC)<\infty.
\end{gather*}
This establishes the Claim. 

To complete the proof, it suffices to show that $C_\m$ is semidualizing over $R_\m$ for all $\m\in\mspec(R)$. 
Let $\m\in\mspec(R)$ be given, and let $\ideal M\in\mspec(S)$ lie over $\m$. 
By assumption, the complex $\Lotimes SC$ is semidualizing over $S$,
so the localization $(\Lotimes SC)_{\ideal M}$ is semidualizing over $S_{\ideal M}$.
That is, the complex
$$\Lotimes[R_\m]{S_{\ideal M}}{(\Lotimes{R_\m}C)}\simeq\Lotimes[R_\m]{S_{\ideal M}}{C_\m}$$
is semidualizing over $S_{\ideal M}$. The induced map $\vf_{\ideal M}\colon R_\m\to S_{\ideal M}$ 
is local and has finite flat dimension, so  $C_\m$ is semidualizing over $R_{\m}$,
by~\cite[Theorem~4.5]{frankild:rrhffd}, as desired.
\end{proof}

In the next result, note that $\Lotimes SM\in\catdb(S)$ holds automatically if $\fd_R(S)<\infty$.

\begin{thm}\label{thm140217a}
Assume that $\vf$ is (locally) of finite flat dimension, and let $M\in\catdb(R)$ be given.
\begin{enumerate}[\rm(a)]
\item\label{thm140217a1}
If $M$ is $\fa$-adically semidualizing with $\Lotimes SM\in\catdb(S)$, then $\Lotimes SM$ is $\fa S$-adically semidualizing over $S$.
\item\label{thm140217a2}
The converse of part~\eqref{thm140217a1} holds if $\supp_R(S)\supseteq\VE(\fa)\bigcap\mspec(R)$ and $M$ is $\fa$-adically finite over $R$,
e.g., if $\vf$ is local and $M$ is $\fa$-adically finite over $R$.
\item\label{thm140217a3}
In particular, the converse of part~\eqref{thm140217a1} holds if $\supp_R(S)\supseteq\VE(\fa)\bigcup\supp_R(X)$ and $\vf$ is flat, e.g., $\vf$ is faithfully flat.
\end{enumerate}
\end{thm}

\begin{proof}
\eqref{thm140217a1}
Assume that $M$ is $\fa$-adically semidualizing.
In particular,  $M$ is $\fa$-adically finite over $R$, so~\cite[Theorem~5.10]{sather:afcc} 
implies that $\Lotimes SM$ is $\fa S$-adically finite over $S$, 
and by~\cite[Theorem~7.3]{sather:elclh} we have an isomorphism in $\catd(\compsa)$
$$\Lotimes[\Comp Ra]{\compsa}{\LLL aM}\simeq\mathbf{L}\widehat\Lambda^{\fa S}(\Lotimes SM).$$
The fact that $\Lotimes SM$ is $\fa S$-adically finite over $S$ tells us that the second displayed complex is in $\catdfb(\compsa)$,
hence so is the first.
Also, Theorem~\ref{lem151005a} implies that 
$\LLL aM$ is semidualizing over $\Comp Ra$.
Since the induced map $\Comp\vf a\colon\Comp Ra\to\compsa$ is locally of finite flat dimension, Lemma~\ref{lem151012b}\eqref{lem151012b1}
implies that
the first displayed complex is semidualizing over $\compsa$, hence so is the second one.
Another application of Theorem~\ref{lem151005a} tells us that $\Lotimes SM$ is $\fa S$-adically semidualizing over $S$, as desired.

\eqref{thm140217a2}
Claim 1: if $\vf$ is local, then we have $\supp_R(S)\supseteq\VE(\fa)\bigcap\mspec(R)$.
Indeed, if $\vf$ is local, then the maximal ideal $\m$ of $R$ satisfies $\m\in\vf^*(\spec(S))$; hence,
we have the second step in the next display.
$$\VE(\fa)\bigcap\mspec(R)=\{\m\}\subseteq\vf^*(\spec(S))=\supp_R(S)$$
The first step is from the assumption $\fa\neq R$, since $R$ is local here. 
The last step is from~\cite[Proposition~5.6(a)]{sather:afcc}.
The establishes  Claim~1. 

Now, to prove part~\eqref{thm140217a2}, assume that $\Lotimes SM$ is $\fa S$-adically semidualizing over $S$, we have
$\supp_R(S)\supseteq\VE(\fa)\bigcap\mspec(R)$, and $M$ is $\fa$-adically finite over $R$.
In particular, we have the isomorphism displayed in the previous paragraph. 
Another application of Theorem~\ref{lem151005a} tells us that 
$\mathbf{L}\widehat\Lambda^{\fa S}(\Lotimes SM)$
is semidualizing over $\compsa$, so the isomorphism implies that
$\Lotimes[\Comp Ra]{\compsa}{\LLL aM}$ is semidualizing over $\compsa$ a well.

Claim 2: $(\Comp\vf a)^*(\spec(\compsa))\supseteq\mspec(\Comp Ra)$.
Indeed, let $\m\in\mspec(\Comp Ra)$. It follows from Lemma~\ref{lem151205a}\eqref{lem151205a1} that
there is a maximal ideal $\m_0\in\mspec(R)\bigcap\VE(\fa)$ such that $\m=\m_0\Comp Ra$.
By assumption, there is a maximal ideal $\ideal M_0\in\mspec(S)$ such that $\m_0=\vf^{-1}(\ideal M_0)$. 
The condition $\m_0\supseteq\fa$ implies that $\ideal M_0\supseteq\m_0S\supseteq\fa S$.
An application of Lemma~\ref{lem151205a}\eqref{lem151205a1} to $S$ shows that the ideal $\ideal M:=\ideal M_0\compsa\supseteq\fa\compsa$
is maximal in $\compsa$ and contracts to $\ideal M_0$ in $S$. 
In particular, the prime ideal $\fP:=(\Comp\vf a)^{-1}(\ideal M)\supseteq\fa\Comp Ra$ contracts to $\m_0$ in $R$.
Since $\fP$ and $\m$ both are  in $\VE(\fa\Comp Ra)$ and contract to $\m_0$ in $R$,
Lemma~\ref{lem151205a}\eqref{lem151205a2} implies that $\fP=\m$.
This establishes Claim~2. 

Now we complete the proof of part~\eqref{thm140217a2}.
Recall that the induced map $\Comp Ra\to\compsa$ is locally of finite flat dimension. 
Since $M$ is $\fa$-adically finite over $R$, we have $\LLL aM\in\catdfb(R)$ and $\supp_R(M)\subseteq\VE(\fa)$.
Thus, Lemma~\ref{lem151012b}\eqref{lem151012b2} conspires with  Claim~2 to imply that
$\LLL aM$ is semidualizing over $\Comp Ra$, and
Theorem~\ref{lem151005a} implies that $M$ is $\fa$-adically semidualizing over $R$, as desired.

\eqref{thm140217a3}
Assume that $\supp_R(S)\supseteq\VE(\fa)\bigcup\supp_R(X)$ and $\vf$ is flat. 
If $\Lotimes SM$ is $\fa S$-adically semidualizing over $S$, then $\Lotimes SM$ is $\fa S$-adically finite over $S$,
so~\cite[Theorem~5.10]{sather:afcc} implies that $M$ is $\fa$-adically finite over $R$.
Thus, the desired converse follows from part~\eqref{thm140217a2}.
\end{proof}

The next result is a bit strange, since completions don't usually interact well with localization.

\begin{thm}\label{cor151122a}
Let $M\in\catdb(R)$ be $\fa$-adically finite. The following  are equivalent.
\begin{enumerate}[\rm(i)]
\item\label{cor151122a1}
$M$ is $\fa$-adically semidualizing.
\item\label{cor151122a2}
For each multiplicatively closed subset $U\subseteq R$ such that $\fa U^{-1}R\neq U^{-1}R$, the $U^{-1}R$-complex
$U^{-1}M\simeq\Lotimes{(U^{-1}R)}M$ is $U^{-1}\fa$-adically semidualizing.
\item\label{cor151122a3}
For all $\p\in\VE(\fa)$, the $R_\p$-complex $M_\p\simeq\Lotimes{R_\p}M$ is $\fa_\p$-adically semidualizing.
\item\label{cor151122a4}
For all $\m\in\VE(\fa)\bigcap\mspec(R)$, the $R_\m$-complex $M_\m\simeq\Lotimes{R_\m}M$ is $\fa_\m$-adically semidualizing.
\end{enumerate}
\end{thm}

\begin{proof}
In view of Theorem~\ref{thm140217a}\eqref{thm140217a1}, it suffices to prove the implication~\eqref{cor151122a4}$\implies$\eqref{cor151122a1}.
Assume that for each $\m\in\VE(\fa)\bigcap\mspec(R)$, the $R_\m$-complex $M_\m$ is $\fa_\m$-adically semidualizing.
Since $M$ is also assumed to be $\fa$-adically finite, to prove that is it $\fa$-adically semidualizing, it suffices by Theorem~\ref{lem151005a}
to show that $\LLL aM$ is semidualizing over $\Comp Ra$.
By assumption, we have $\LLL aM\in\catdfb(\Comp Ra)$, so to show that this complex is semidualizing over $\Comp Ra$,
it suffices to show that for each $\ideal M\in\mspec(\Comp Ra)$ the localization $\LLL aM_{\ideal M}$ is semidualizing over $(\Comp Ra)_{\ideal M}$; see~\cite[Lemma~2.3]{frankild:rrhffd}.

Let $\ideal M\in\mspec(\Comp Ra)$ be given, and 
let $\m$ be the contraction of $\ideal M$ in $R$.
By assumption, the $R_\m$-complex $M_\m$ is $\fa_\m$-adically semidualizing.
Thus, Theorem~\ref{lem151005a} implies that the complex 
$\mathbf{L}\widehat{\Lambda}^{\fa_\m}(M_\m)$ is semidualizing over $\widehat{R_\m}^{\fa_\m}$.

According to~\cite[Corollaire~0.7.6.14]{grothendieck:ega1}, the natural map 
$\Comp Ra\to\widehat{R_\m}^{\fa_\m}$ is flat. 
The ring $\widehat{R_\m}^{\fa_\m}$ is local with maximal ideal $\m \widehat{R_\m}^{\fa_\m}$, and the contraction of this maximal ideal in
$\Comp Ra$ is $\m\Comp Ra=\ideal M$.
It follows that the induced map $(\Comp Ra)_{\ideal M}\to\widehat{R_\m}^{\fa_\m}$ is flat and local. 
As we have already seen, the next  complex  is semidualizing over $\widehat{R_\m}^{\fa_\m}$.
\begin{align*}
\mathbf{L}\widehat{\Lambda}^{\fa_\m}(M_\m)
&\simeq\mathbf{L}\widehat{\Lambda}^{\fa R_\m}(\Lotimes{R_\m}M)\\
&\simeq\Lotimes[\Comp Ra]{\widehat{R_\m}^{\fa R_\m}}{\LLL aM}\\
&\simeq\Lotimes[(\Comp Ra)_{\ideal M}]{\widehat{R_\m}^{\fa R_\m}}{(\Lotimes[\Comp Ra]{(\Comp Ra)_{\ideal M}}{\LLL aM})}\\
&\simeq\Lotimes[(\Comp Ra)_{\ideal M}]{\widehat{R_\m}^{\fa_\m}}{\LLL aM_{\ideal M}}
\end{align*}
The second isomorphism here is by~\cite[Theorem~7.3]{sather:elclh} with $S=R_\m$.
Since we have $\LLL aM_{\ideal M}\in\catdfb(\widehat{R_\m}^{\fa_\m})$, it follows from~\cite[Theorem~5.6]{christensen:scatac} that $\LLL aM_{\ideal M}$
is semidualizing over $\widehat{R_\m}^{\fa_\m}$as desired.
\end{proof}

\subsection*{Extended Derived Local Cohomology}
We next consider the behavior of adic semidualizing complexes with respect to the functor $\RRGno a$.

\begin{prop} \label{prop150619b}
An $R$-complex $M\in\catdb(R)$ is $\fa$-adically semidualizing over $R$
if and only if $\supp_R(M)\subseteq\VE(\fa)$ and
$\RRG aM$ is $\fa\Comp Ra$-adically semidualizing over $\Comp Ra$.
\end{prop}

\begin{proof}
For the forward implication, assume that $M$ is $\fa$-adically semidualizing over $R$.
In particular, we have $\supp_R(M)\subseteq\VE(\fa)$, hence $M\simeq\RG aM$ by Fact~\ref{cor130528a}.
From Fact~\ref{fact130619b}
we therefore have 
isomorphisms
$$\RRG aM\simeq \Lotimes{\Comp Ra}{\RG aM}\simeq \Lotimes{\Comp Ra}M$$ 
in $\catd(\Comp Ra)$.
Thus, Theorem~\ref{thm140217a}\eqref{thm140217a1} 
implies that $\RRG aM$ is $\fa\Comp Ra$-adically semidualizing over $\Comp Ra$, as desired.

For the converse, assume that $\supp_R(M)\subseteq\VE(\fa)$ and that $\RRG aM$ is $\fa\Comp Ra$-adically semidualizing over $\Comp Ra$.
The condition $\supp_R(M)\subseteq\VE(\fa)$ implies that we have
$M\simeq \RG aM\simeq Q(\RRG aM)$ in $\catd(R)$, by Facts~\ref{cor130528a} and~\ref{fact130619b}, respectively.
Proposition~\ref{prop151011a} implies that $M$ is $\fa$-adically semidualizing over $R$, as desired.
\end{proof}

\begin{disc}\label{disc151204c}
Some of our results become trivial when $R$ is $\fa$-adically complete. 
For instance, if $R$ is $\fa$-adically complete, then the conclusion of the previous result says that $M$ is $\fa$-adically semidualizing
if and only if $\supp_R(M)\subseteq\VE(\fa)$ and $M$ is $\fa$-adically semidualizing.
Indeed, the completeness assumption implies  $\RRGno a=\RGno a$; so
if $\supp_R(M)\subseteq\VE(\fa)$, e.g., if $M$ is $\fa$-adically semidualizing, then this says that $\RRG aM\simeq\RG aM\simeq M$
by Fact~\ref{cor130528a}. Similar comments apply to our next result.

On the other hand, other results of this section have cleaner (and non-trivial) statements when one assumes that $R$ is $\fa$-adically complete. 
We include a few of these explicitly below.
\end{disc}

Our next result contains part of Theorem~\ref{thm151203c} from the introduction.

\begin{thm}\label{thm151012a}
The functor $\RRGno a$ induces a bijection $\mathfrak{S}^\fa(R)\to\s^{\fa\Comp Ra}(\Comp Ra)$ with inverse induced by
the forgetful functor
$Q\colon\catd(\Comp Ra)\to\catd(R)$.
Also, the bijection $\RRGno a\colon\mathfrak{S}^\fa(R)\to\s^{\fa\Comp Ra}(\Comp Ra)$
is the same as the base-change map 
$\Lotimes{\Comp Ra}-\colon\mathfrak{S}^\fa(R)\to\s^{\fa\Comp Ra}(\Comp Ra)$
from Theorem~\ref{thm140217a}.
\end{thm}

\begin{proof}
Propositions~\ref{prop151011a} and~\ref{prop150619b} show that 
$Q$ and $\RRGno a$ induce well-defined maps $\s^{\fa\Comp Ra}(\Comp Ra)\to\mathfrak{S}^\fa(R)$
and $\mathfrak{S}^\fa(R)\to\s^{\fa\Comp Ra}(\Comp Ra)$.
Also,  $Q$ and $\RRGno a$ induce
inverse equivalences between $\catdator$ and $\catdaator$,
by~\cite[Theorem~4.12]{sather:elclh}; as these contain the $\fa$-adic semidualizing $R$-complexes
and the $\fa \Comp Ra$-adic semidualizing $\Comp Ra$-complexes, respectively, the  maps 
$\s^{\fa\Comp Ra}(\Comp Ra)\to\mathfrak{S}^\fa(R)$
and $\mathfrak{S}^\fa(R)\to\s^{\fa\Comp Ra}(\Comp Ra)$
are inverse bijections.
Lastly, the proof of Proposition~\ref{prop150619b} shows that the maps $\mathfrak{S}^\fa(R)\to\s^{\fa\Comp Ra}(\Comp Ra)$
induced by $\RGno a$ and $\Lotimes{\Comp Ra}-$ are equal.
\end{proof}

\begin{cor}\label{prop151013a}
Assume that $\vf$ is flat  and that the induced map $\Comp Ra\to\comp S^{\fa S}$ is an isomorphism.
Then the forgetful functor $Q$ induces a bijection $\s^{\fa S}(S)\to\s^{\fa}(R)$ with inverse induced by
the functor $\Lotimes S-$.
\end{cor}

\begin{proof}
Proposition~\ref{prop151011a} and Theorem~\ref{thm140217a}\eqref{thm140217a1} show that the functors $Q$ and $\Lotimes S-$
induce well-defined functions $\s^{\fa S}(S)\to\s^{\fa}(R)\to\s^{\fa S}(S)$.
Also, the next diagrams commute, where the unspecified maps are given by the respective forgetful functors.
$$\xymatrix@C=15mm{
\s^{\fa}(R)\ar[r]^-{\Lotimes S-}\ar[d]_{\Lotimes{\Comp Ra}-}^\simeq
&\s^{\fa S}(S)\ar[d]^{\Lotimes[S]{\compsa}-}_\simeq
\\
\s^{\fa\Comp Ra}(\Comp Ra)\ar[r]^-{\Lotimes[\Comp Ra]{\compsa}{-}}_-\simeq
&\s^{\fa\compsa}(\compsa)
}
\qquad
\xymatrix@C=15mm{
\s^{\fa S}(S)\ar[r]^-{Q}
&\s^{\fa}(R)
\\
\s^{\fa\compsa}(\compsa)\ar[r]_-\simeq\ar[u]^\simeq
&\s^{\fa\Comp Ra}(\Comp Ra)\ar[u]^\simeq
}$$
The vertical bijections are from Theorem~\ref{thm151012a}, 
and the horizontal ones are from our completion assumption.
It follows that the upper horizontal maps are bijections as well. 
Furthermore, since three of the four forgetful maps are the inverses of the corresponding base change maps,
one uses the diagrams to show  that the upper horizontal maps compose to the respective identities, as desired.
\end{proof}

For perspective in the next result, recall that~\cite[Theorem~6.1]{sather:afcc} shows that every $\fa$-adically finite $R$-complex $X$ satisfies $\fd_R(X)=\pd_R(X)$.

\begin{cor}\label{cor151013a}
Assume  $\vf$ is flat  and  the induced map $\Comp Ra\to\comp S^{\fa S}$ is bijective.
\begin{enumerate}[\rm(a)]
\item\label{cor151013a1}
Given an $\fa$-adic semidualizing complex $M\in\s^{\fa}(R)$, we have
\begin{align*}
\id_R(M)&=\id_S(\Lotimes SM)
&\fd_R(M)&=\fd_S(\Lotimes SM)
\end{align*}
\item\label{cor151013a2}
Given an $\fa S$-adic semidualizing complex $N\in\s^{\fa S}(S)$, we have
\begin{align*}
\id_R(Q(N))&=\id_S(N)
&\fd_R(Q(N))&=\fd_S(N)
\end{align*}
\end{enumerate}
\end{cor}

\begin{proof}
Note that the fact that the map $\Comp Ra\to\compsa$ is an isomorphism implies that the 
same is true of the lower horizontal map in the next commutative diagram
$$\xymatrix{
R/\fa \ar[r]\ar[d]_\cong
&S/\fa S\ar[d]^{\cong}
\\
\Comp Ra/\fa \Comp Ra\ar[r]^\cong
&\compsa/\fa \compsa
}$$
It follows that the upper horizontal map is also an isomorphism.
Consequently, for each $\p\in\VE(\fa)$, the induced map $R/\p\to S/\p S$ is an isomorphism,
hence so is the map $\kappa(\p)\to\Otimes S{\kappa(\p)}$.

Let $M\in\s^{\fa}(R)$ and $N\in\s^{\fa S}(S)$ be given.
We prove the injective dimension formulas; the flat dimension formulas are verified similarly.

The inequality $\id_R(Q(N))\leq\id_S(N)$ holds because the flatness of $\vf$ implies that any
(bounded) semi-injective resolution of $N$ over $S$ restricts to a (bounded) semi-injective resolution of $Q(N)$ over $R$. 

Next, we verify the inequality $\id_S(\Lotimes SM)\leq\id_R(M)$.
For this argument, assume without loss of generality that $\id_R(M)<\infty$.
Then the minimal semi-injective resolution $M\xra\simeq J$ over $R$ is  bounded with minimal length,
since it is a direct summand of every semi-injective resolution of $M$. 
Furthermore, each module $J_i$ is a direct sum of $R$-modules of the form $E_R(R/\p)$ with $\p\in\supp_R(M)\subseteq\VE(\fa)$;
see~\cite[Proposition~3.8]{sather:scc}. Since $\vf$ is flat, we have $\Lotimes SM\simeq\Otimes SJ$ in $\catd(S)$,
so it suffices to show that each module $\Otimes S{J_i}$ is injective over $S$.
This follows from~\cite[Theorem~1]{foxby:imufbc}:
for each $\p\in\ass_R(J_i)\subseteq\supp_R(M)\subseteq\VE(\fa)$, the map $\kappa(\p)\to\Otimes S{\kappa(\p)}$ is an isomorphism, 
so the $S$-module $\Otimes S{J_i}$ is injective.

By Corollary~\ref{prop151013a}, we have $Q(\Lotimes SM)\simeq M$, so the previous two paragraphs (with $N=\Lotimes SM$)
imply that 
$$\id_R(M)=\id_R(Q(\Lotimes SM))\leq\id_S(\Lotimes SM)\leq\id_R(M)$$
so we have the first formula in part~\eqref{cor151013a1}.
The first formula in part~\eqref{cor151013a2} follows similarly using the isomorphism $N\simeq\Lotimes S{(Q(N))}$
from Corollary~\ref{prop151013a}.
\end{proof}

For convenience, we state the special case $S=\Comp Ra$ of the previous result next.

\begin{cor}\label{cor151013b}
Consider the forgetful functor
$Q\colon\catd(\Comp Ra)\to\catd(R)$.
\begin{enumerate}[\rm(a)]
\item\label{cor151013b1}
Given an $\fa$-adic semidualizing complex $M\in\s^{\fa}(R)$, we have
\begin{align*}
\id_R(M)&=\id_{\Comp Ra}(\RRG aM)
&\fd_R(M)&=\fd_{\Comp Ra}(\RRG aM)
\end{align*}
\item\label{cor151013b2}
Given an $\fa {\Comp Ra}$-adic semidualizing complex $N\in\s^{\fa {\Comp Ra}}({\Comp Ra})$, we have
\begin{align*}
\id_R(Q(N))&=\id_{\Comp Ra}(N)
&\fd_R(Q(N))&=\fd_{\Comp Ra}(N)
\end{align*}
\end{enumerate}
\end{cor}

\begin{proof}
By Theorem~\ref{thm151012a}, we have $\Lotimes{\Comp Ra}M\simeq\RRG aM$,
so the desired conclusions follow from
Corollary~\ref{cor151013a}.
\end{proof}

\subsection*{Extended Derived Local Homology}

We now investigate the interaction between the adic semidualizing property and the functor $\LLLno a$, building on Theorem~\ref{lem151005a}.
Our next result contains the rest of Theorem~\ref{thm151203c} from the introduction.

\begin{thm}\label{thm140204a}
Consider the forgetful functor
$Q\colon\catd(\Comp Ra)\to\catd(R)$.
Then $\LLLno a$ induces a bijection $\mathfrak{S}^\fa(R)\to\s(\Comp Ra)$ with inverse induced by
$Q\circ\RGano a\simeq\RGno a\circ Q$.
\end{thm}

\begin{proof}
The isomorphism $Q\circ\RGano a\simeq\RGno a\circ Q$ is from~\cite[Corollary~4.14]{sather:elclh}.

By Theorem~\ref{lem151005a},
the functor $\LLLno a$ induces a well-defined function $\mathfrak{S}^\fa(R)\to\s(\Comp Ra)$.
On the other hand, Corollary~\ref{thm130330b} and Proposition~\ref{prop151011a} show that the functors
$\RGano a$ and $Q$ induce well-defined functions $\s(\Comp Ra)\to\s^{\fa\Comp Ra}(\Comp Ra)\to\s^\fa(R)$.

Furthermore, from~\cite[Theorem~6.3(b)]{sather:elclh}, the functor $\LLLno a$ induces an equivalence 
between the category of $\fa$-adically finite $R$-complexes and the category $\catdfb(\Comp Ra)$,
with quasi-inverse induced by $Q\circ\RGano a$.
Since the $\fa$-adically semidualizing $R$-complexes are $\fa$-adically finite over $R$, and the semidualizing $\Comp Ra$-complexes
are in $\catdfb(\Comp Ra)$, the maps from the previous paragraph are inverse bijections.
\end{proof}

\begin{cor}\label{thm140204az}
Assume that $R$ is $\fa$-adically complete.
The functor $\LLno a$ induces a bijection $\mathfrak{S}^\fa(R)\to\s(R)$ with inverse induced by
$\RGno a$.
\end{cor}

\begin{disc}\label{disc151015a}
Assume that $\vf$ is flat  and that the induced map $\Comp Ra\to\comp S^{\fa S}$ is an isomorphism.
The following diagram displays the bijections described in 
Theorem~\ref{thm151012a}, Corollary~\ref{prop151013a}, and Theorem~\ref{thm140204a};
in it, each pair of arrows is an inverse pair, and each cell commutes (in every combination).
In a feeble attempt to keep the notation from getting out of hand, we use $Q$ for each of the forgetful functors.
$$\xymatrix@C=14mm{
\s^{\fa}(R)\ar@<0.5ex>[rrr]^-{\Lotimes S-}\ar@<0.5ex>[dd]^-{\LLLno a}\ar@<0.5ex>[rd]^>>>>>{\ \ \RRGno a=\Lotimes{\Comp Ra}-}
&&&\s^{\fa S}(S)\ar@<0.5ex>[lll]^{Q}
\ar@<-0.5ex>[dl]_>>>>>{\mathbf{R}\widehat{\Gamma}_{\fa S}=\Lotimes[S]{\compsa}{-}\ \ }
\ar@<-0.5ex>[dd]_-{\mathbf{L}\widehat{\Lambda}^{\fa S}} 
\\
&\s^{\fa\Comp Ra}(\Comp Ra)\ar@<0.5ex>[ul]^-{Q}\ar@<0.5ex>[ld]^<<<<<{\LLano a}\ar@<0.5ex>[r]^-{\Lotimes[\Comp Ra]{\compsa}-}
&\s^{\fa\compsa}(\compsa) \ar@<-0.5ex>[ur]_-{Q}\ar@<-0.5ex>[rd]_<<<<<{\mathbf{L}\Lambda^{\fa \compsa}\ \ }\ar@<0.5ex>[l]^-{Q}
\\
\s(\Comp Ra)\ar@<0.5ex>[uu]^>>>>>>>>>>>{Q\circ\RGano a}^<<<<<<<<<{=\RGno a\circ Q}\ar@<0.5ex>[ur]^-{\RGano a} 
\ar@<0.5ex>[rrr]^-{\Lotimes[\Comp Ra]{\compsa}-}
&&&\s(\compsa)\ar@<0.5ex>[lll]^-{Q}\ar@<-0.5ex>[uu]_>>>>>>>>>>>{Q\circ\mathbf{R}\Gamma_{\fa\compsa}}_<<<<<<<<<{=\mathbf{R}\Gamma_{\fa S}\circ Q}\ar@<-0.5ex>[ul]_-{\mathbf{R}\Gamma_{\fa\compsa}}
}$$
\end{disc}

We end this subsection with connections to tilting and dualizing complexes.

\begin{cor}\label{cor150525a}
Consider the forgetful functor
$Q\colon\catd(\Comp Ra)\to\catd(R)$.
Let $M\in\s^a(R)$ and $C\in\s(\Comp Ra)$ be given.
\begin{enumerate}[\rm(a)]
\item \label{cor150525a1}
We have $\fd_R(M)<\infty$ if and only if 
$\LLL aM$ is a tilting $\Comp Ra$-complex.
\item \label{cor150525a2}
We have $\id_R(M)<\infty$ if and only if 
$\LLL aM$ is a dualizing $\Comp Ra$-complex.
\item \label{cor150525a3}
The complex $C$ is tilting over $\Comp Ra$
if and only if 
$\fd_R(Q(\RGa aC))<\infty$.
\item \label{cor150525a4}
The complex $C$ is dualizing over $\Comp Ra$
if and only if 
$\id_R(Q(\RGa aC))<\infty$.
\end{enumerate}
\end{cor}

\begin{proof}
By Theorem~\ref{thm140204a}, the complex $\LLL aM$ 
is semidualizing over $\Comp Ra$ such that $M\simeq Q(\RGa a{\LLL aM})\simeq\RG a{Q(\LLL aM)}$,
and  the complex $Q(\RGa aC)\simeq\RG a{Q(C)}$ 
is $\fa$-adically semidualizing over $R$ such that $C\simeq \LLL a{Q(\RGa aC)}$.

In this paragraph, we assume  that $\fd_R(M)<\infty$ and show that $\pd_{\Comp Ra}(\LLL aM)<\infty$.
From~\cite[Proposition~5.3(b)]{sather:elclh}, we have
$\fd_{\Comp Ra}(\LLL aM)\leq\fd_R(M)<\infty$. 
Since $M$ is $\fa$-adically finite, the complex
$\LLL aM$ is homologically finite over $\Comp Ra$, and it follows that $\pd_{\Comp Ra}(\LLL aM)<\infty$, as desired.

Next, we assume that $\pd_{\Comp Ra}(C)<\infty$ and show that $\fd_R(Q(\RGa aC))<\infty$. 
Since $\Comp Ra$ is flat over $R$, we have $\fd_R(Q(C))\leq\pd_{\Comp Ra}(C)<\infty$.
Thus, the following $R$-complexes have finite flat dimension over $R$
$$Q(\RGa aC)\simeq \RG a{Q(C)}\simeq\Lotimes{\RG aR}Q(C)$$
since $\RG aR$ and $Q(C)$ both have finite flat dimension.

As in the proof of Corollary~\ref{cor151013a}, 
parts~\eqref{cor150525a1} and~\eqref{cor150525a3} now follow. 
Parts~\eqref{cor150525a2} and~\eqref{cor150525a4} are verified similarly.
\end{proof}

\begin{cor}\label{cor150525az}
Assume that $R$ is $\fa$-adically complete, and let $M\in\s^a(R)$ and $C\in\s(R)$ be given.
\begin{enumerate}[\rm(a)]
\item \label{cor150525az1}
We have $\fd_R(M)<\infty$ if and only if 
$\LL aM$ is a tilting $R$-complex.
\item \label{cor150525az2}
We have $\id_R(M)<\infty$ if and only if 
$\LL aM$ is a dualizing $R$-complex.
\item \label{cor150525az3}
The complex $C$ is tilting over $R$
if and only if 
$\fd_R(\RG aC)<\infty$.
\item \label{cor150525az4}
The complex $C$ is dualizing over $R$
if and only if 
$\id_R(\RG aC)<\infty$.
\end{enumerate}
\end{cor}

\subsection*{Semidualizing DG $K$-Modules}
Given an $\fa$-adic semidualizing $R$-complex $M$, 
a crucial point in Theorem~\ref{thm140204a} is that we have $\LLL aM\in\catdfb(R)$.
By Fact~\ref{thm130612a}, we also have $\Lotimes KM\in\catdfb(R)$;
when translated to the language of ``DG $K$-modules'', this says $\Lotimes KM\in\catdfb(R)$.
In short, this means that, when one considers the exterior algebra structure on $K$,
the complex $\Lotimes KM\simeq\Otimes KM$ inherits a $K$-module structure from the left;
the DG structure means that this scalar multiplication respects the differentials in these complexes. 

In this setting, one forms the derived category $\catd(R)$ from the category of DG $K$-modules
like one forms $\catd(R)$ from the category of $R$-complexes. 
A DG $K$-module $N$ is in $\catdfb(K)$ provided that $\bigoplus_{i\in\bbz}\HH_i(N)$ is finitely generated over $R$,
that is, when $N$ is in $\catdfb(R)$. And $N$ is a \emph{semidualizing} DG $K$-module when it is in $\catdfb(K)$
and the natural homothety morphism $K\to\Rhom[K] NN$ is an isomorphism in $\catd(K)$. 
The set of shift-isomorphism classes of semidualizing DG $K$-modules is denoted $\s(K)$.
See~\cite{christensen:dvke, nasseh:lrfsdc, nasseh:lrec3, nasseh:ldgm} for more about these objects, 
including applications to the study of $\s(R)$. 

We now reach the point of this discussion: in the same way that the condition $\LLL aM\in\catdfb(\Comp Ra)$ makes it reasonable for us to have
$\LLL aM\in\s(\Comp Ra)$, the condition $\Lotimes KM\in\catdfb(K)$ makes it reasonable for us to have $\Lotimes KM\in\s(K)$, as we see in the next result.

\begin{cor}\label{cor151102a}
The functor $\Lotimes K-$ induces a bijection
$\s^\fa(R)\to\s(K)$.
\end{cor}

\begin{proof}
Set $\comp K:=\Lotimes{\Comp Ra}K$, which is the Koszul complex over $\Comp Ra$ on a finite generating sequence for $\fa\Comp Ra$. 

Theorem~\ref{thm140204a}, implies that the functor $\LLL a-\colon\catd(R)\to\catd(\Comp Ra)$ induces a bijection
$\s^\fa(R)\to\s(\Comp Ra)$.
From~\cite[Corollary~3.10]{nasseh:ldgm}, we know that the functor $\Lotimes[\Comp Ra]{\comp K}-\colon\catd(\Comp Ra)\to\catd(\comp K)$ induces a bijection
$\s(\Comp Ra)\to\s(\comp K)$.
Also, since the natural map $K\to\comp K$ is a quasiisomorphism of DG algebras, the forgetful functor
$Q\colon\catd(\comp K)\to\catd(K)$ induces a bijection $\s(\comp K)\to\s(K)$. 
Thus, it remains to show that the composition of these bijections 
$\s^\fa(R)\to\s(\Comp Ra)\to\s(\comp K)\to\s(K)$ is given by $\Lotimes K-$.

Let $M\in\s^\fa(R)$.
We need to show that $Q(\Lotimes[\Comp Ra]{\comp K}{\LLL aM})\simeq\Lotimes KM$ in $\catd(K)$. 
This is accomplished in the next sequence, wherein
$Q'\colon\catd(\Comp Ra)\to\catd(R)$ denotes the forgetful functor. 
\begin{align*}
Q(\Lotimes[\Comp Ra]{\comp K}{\LLL aM})
&\simeq Q(\Lotimes[\Comp Ra]{(\Lotimes K{\Comp Ra})}{\LLL aM})\\
&\simeq\Lotimes K{Q'(\LLL aM)} \\
&\simeq\Lotimes K{\LL aM}\\
&\simeq\Lotimes KM
\end{align*}
The first two isomorphisms here are straightforward, and
the third one is from Fact~\ref{fact130619b}.
For the fourth one, note that~\cite[Lemma~2.8]{sather:elclh} shows that the natural morphism $\Lotimes KM\to\Lotimes K{\LL aM}$ is an isomorphism
in $\catd(R)$. Since it is also a morphism in $\catd(K)$, it is also an isomorphism in $\catd(K)$, as desired.
\end{proof}

\begin{disc}\label{disc151102a}
Unlike in our previous results, it is not clear how to give a  functorial description of the inverse of the 
bijection $\s^\fa(R)\to\s(K)$ from Corollary~\ref{cor151102a}.
The problem is that~\cite[Corollary~3.10]{nasseh:ldgm} uses a lifting property to show that the map $\s(\Comp Ra)\to\s(\comp K)$ is bijective,
but it does not give a functorial description of  the inverse of this map, nor is it clear that such a description exists.
\end{disc}

\section{Dualizing Complexes and Gorenstein Rings}\label{sec151204a}

We begin this section by proving Theorem~\ref{thm151203a} from the introduction.

\begin{thm}\label{cor150525axx}
\begin{enumerate}[\rm(a)]
\item \label{cor150525axx3}
The ring $\Comp Ra$ has a dualizing complex if and only if $R$ has an $\fa$-adically semidualizing complex of finite injective dimension.
\item \label{cor150525axx1}
If $(R,\m,k)$ is local, then the $\Comp Rm$-complex $\LLL m{E_R(k)}$ is dualizing for $\Comp Rm$.
\end{enumerate}
\end{thm}

\begin{proof}
\eqref{cor150525axx3}
For one implication, if $\Comp Ra$ has a dualizing complex $D$, then
Theorem~\ref{thm140204a} and Corollary~\ref{cor150525a}\eqref{cor150525a4} imply that $M:=Q(\RGa aD)$ is an $\fa$-adically semidualizing $R$-complex
of finite injective dimension. 
Conversely, if $M$ is an $\fa$-adically semidualizing $R$-complex with $\id_R(M)<\infty$,
then Corollary~\ref{cor150525a}\eqref{cor150525a2} implies that the complex $\LLL aM$ is dualizing over $\Comp Ra$.

\eqref{cor150525axx1}
When $(R,\m,k)$ is local, the injective hull $E:=E_R(k)$ is $\m$-adically semidualizing by Proposition~\ref{prop120925b}.
Since it also has finite injective dimension over $R$, the desired conclusion follows from 
Corollary~\ref{cor150525a}\eqref{cor150525a2} as in the previous paragraph.

Alternately, one can prove this using Grothendieck's local duality, appropriately extended. Indeed,
in the following display, the first isomorphism in $\catd(\Comp Rm)$ is from Fact~\ref{fact130619b},
and the second one is from~\cite[Lemma~1.5(a)]{kubik:hamm1}.
\begin{align*}
\LLL mE
&\simeq\Rhom{\RRG mR}{E} \\
&\simeq\Rhom[\Comp Rm]{\RRG mR}{E}\\
&\simeq\Rhom[\Comp Rm]{\RGa m{\Comp Rm}}{E}
\end{align*}
The third isomorphism is from~\cite[Lemmas~4.4--4.5]{sather:elclh}.
Now,  local duality over $\Comp Rm$ allows us to conclude that the last complex in this display is dualizing for $\Comp Rm$.
\end{proof}

\begin{cor}\label{cor150525axxz}
\begin{enumerate}[\rm(a)]
\item \label{cor150525axxz3}
An $\fa$-adically complete ring has a dualizing complex if and only if it has an $\fa$-adically semidualizing complex of finite injective dimension.
\item \label{cor150525axxz1}
If $(R,\m,k)$ is local and $\m$-adically complete, then the $R$-complex $\LL m{E_R(k)}$ is dualizing for $R$.
\end{enumerate}
\end{cor}

\begin{disc}\label{disc151012b}
It is important to note that Theorem~\ref{cor150525axx}\eqref{cor150525axx3} cannot be used to construct dualizing complexes
for rings that don't have them, obviously. The point is that the condition of $R$ having an $\fa$-adic semidualizing complex of finite injective dimension
can be quite restrictive, in general. 

Our alternate proof of Theorem~\ref{cor150525axx}\eqref{cor150525axx1} 
uses the fact that $\Comp Rm$ has a dualizing complex, since that is part of local duality. On the other hand,
the first proof we give for this result does not use this fact, so it gives a new proof of the
existence of a dualizing complex
for $\Comp Rm$.

Also, from Fact~\ref{fact130619b} we have the next isomorphism in $\catd(\Comp Rm)$
$$\LLL mE\simeq\Rhom{\Comp Rm}{\LL mE}$$
so this gives another strange description of a dualizing complex for $\Comp Rm$.

This result also shows a stark distinction between the functors
$\LLLno a$ and $\Lotimes{\Comp Ra}-$. Indeed, the complex
$\LLL m{E}$ is dualizing for $\Comp Rm$; in particular, it is homologically finite over $\Comp Rm$.
On the other hand, we have $\Lotimes{\Comp Rm}{E}\simeq E_{\Comp Rm}(k)$, which is only homologically finite over $\Comp Rm$
if $R$ is artinian. Even when $R$ is $\m$-adically complete, this shows how strange the functor $\LLno a$ is,
for instance, since $E$ is a module, but $\LL m{E}$ is a dualizing complex for $R$, by
Corollary~\ref{cor150525axxz}\eqref{cor150525axxz1}
\end{disc}

We now turn our attention to a uniqueness result for  Gorenstein rings. 
The next result and its corollary should be compared to~\cite[Corollary~8.6]{christensen:scatac},
which says that the  semidualizing complexes over a local Gorenstein ring $R$ are exactly the complexes of the form
$\shift^nR$ for some $n$. 

\begin{cor}\label{cor130723b}
Assume that $\Comp Ra$ is locally Gorenstein, e.g., that $R$ is  locally Gorenstein.
Consider the forgetful functor
$Q\colon\catd(\Comp Ra)\to\catd(R)$. 
\begin{enumerate}[\rm(a)]
\item\label{cor130723b1}
The $\fa$-adically semidualizing $R$-complexes
are precisely the $R$-complexes of the form $Q(\RGa aL)$ for some tilting $\Comp Ra$-complex $L$.
\item\label{cor130723b2}
Assume that $(R,\m,k)$ is local. Then the $\fa$-adically semidualizing $R$-complexes
are precisely the $R$-complexes of the form $\shift^n\RG aR$ for some $n\in\bbz$.
In particular, the $\m$-adically semidualizing complexes 
are precisely the $R$-complexes of the form $\shift^n E_R(k)$ for some $n\in\bbz$.
\end{enumerate}
\end{cor}

\begin{proof}
Lemma~\ref{lem151205a}\eqref{lem151205a4} shows that if $R$ is locally Gorenstein, then so is $\Comp Ra$.

\eqref{cor130723b1}
In view of Theorem~\ref{thm140204a}, it suffices to show that every semidualizing $\Comp Ra$-complex $C$ is tilting.
For each $\ideal P\in\spec(\Comp Ra)$,
the $\Comp Ra_{\ideal P}$ complex $C_{\ideal P}$ is semidualizing, hence it is isomorphic to $\shift^n\Comp Ra_{\ideal P}$
for some $n$ by~\cite[Corollary~8.6]{christensen:scatac}. It follows from~\cite[Proposition~4.4]{frankild:rbsc}
that $C$ is tilting over $\Comp Ra$, as desired.

\eqref{cor130723b2}
In the following sequence of isomorphisms in $\catd(R)$, the first isomorphism is from Fact~\ref{fact130619b},
and the second  one is from~\cite[Corollary~4.14]{sather:elclh}.
\begin{align*}
Q(\RGa a{\Comp Ra})
&\simeq Q(\RGa a{\LLL aR})\\
&\simeq \RG a{Q(\LLL aR)}\\
&\simeq \RG a{\LL aR}\\
&\simeq\RG aR
\end{align*}
The third isomorphism is from Fact~\ref{fact130619b}, 
and the last one is MGM equivalence~\ref{fact150626a}.

Since $R$ is local and Gorenstein, the same is true of $\Comp Ra$,
so~\cite[Corollary~8.6]{christensen:scatac} implies that the only semidualizing $\Comp Ra$-complex, up to isomorphism
and shift, is $\Comp Ra$. Thus, 
part~\eqref{cor130723b1}
and the previous paragraph show that the $\fa$-adically semidualizing $R$-complexes
are of the form $\shift^n\RG aR$.
In particular, for the ideal $\fa=\m$, 
the $\fa$-adically semidualizing $R$-complexes
are of the form $\shift^n\RG mR\simeq\shift^{n-d}E_R(k)$
where $d=\dim(R)$. (This uses the well-known structure of the minimal injective resolution of $R$.)
\end{proof}

\begin{cor}\label{cor130723bz}
Assume that $R$ is  locally Gorenstein
and $\fa$-adically complete. 
Then the $\fa$-adically semidualizing $R$-complexes
are precisely the $R$-complexes of the form $\RG aL$ for some tilting $R$-complex $L$.
\end{cor}
\begin{disc}
\label{disc151025a}
One can combine Theorem~\ref{thm140204a} with other results from the semidualizing literature to obtain
further results like Corollary~\ref{cor130723b}.
For instance, if $R$ is local, then we know from~\cite[Theorem~A]{nasseh:lrfsdc} that $\s(\Comp Ra)$ is finite, 
so we conclude that $\s^{\fa}(R)$ is finite as well. 
If $R$ is local and either is Golod or has embedding codepth at most 3, then~\cite[Theorem~A]{nasseh:lrec3} 
shows that $|\s(\Comp Ra)|\leq 2$, so we have
$|\s^{\fa}(R)|\leq 2$.
\end{disc}

Here is the example promised in~\cite[Example~6.4]{sather:elclh}.

\begin{ex}\label{ex150628a}
Let $(R,\m,k)$ be a local ring.
The injective hull $E:=E_R(k)$ is $\m$-adically semidualizing by Proposition~\ref{prop120925b}\eqref{prop120925b3}.
Suppose that there is an $R$-complex $N\in\catdfb(R)$ such that $E\simeq\RG mN$. 
We show that $N$ is dualizing for $R$.
It suffices to show that the $\Comp Rm$-complex $\Lotimes{\Comp Rm}N\simeq\LLL mN$ is dualizing for $\Comp Rm$;
see~\cite[(2.2)]{avramov:rhafgd} and Fact~\ref{fact130619b}.
From Theorem~\ref{cor150525axx}\eqref{cor150525axx1}, the $\Comp Rm$-complex $\LLL mE$ is dualizing for $\Comp Rm$,
so it suffices to show that $\LLL mN\simeq\LLL mE$.
To this end, we compute:
\begin{align*}
\LLL mE
&\simeq\LLL m{\RG mN}
\simeq\LLL mN.
\end{align*}
The assumption $E\simeq\RG mN$ explains the first isomorphism in this sequence,
and the second one is from~\cite[Lemma~4.1]{sather:elclh}.
\end{ex}

\section*{Acknowledgments}
We are grateful to Srikanth Iyengar, 
Liran Shaul,
and Amnon Yekutieli
for helpful comments about this work.

\providecommand{\bysame}{\leavevmode\hbox to3em{\hrulefill}\thinspace}
\providecommand{\MR}{\relax\ifhmode\unskip\space\fi MR }
\providecommand{\MRhref}[2]{%
  \href{http://www.ams.org/mathscinet-getitem?mr=#1}{#2}
}
\providecommand{\href}[2]{#2}

\end{document}